\documentclass[final,onefignum,onetabnum]{siamonline190516}
\usepackage{lipsum}
\usepackage{amsfonts}
\usepackage{graphicx}
\usepackage{epstopdf}
\usepackage{algorithmic}

\usepackage[utf8]{inputenc}
\usepackage{verbatim}

\usepackage{subcaption}
\usepackage{graphbox}
\usepackage{amsmath,amssymb}
\usepackage{hyperref}

\usepackage{bm}
\usepackage{enumitem}

\ifpdf
  \DeclareGraphicsExtensions{.eps,.pdf,.png,.jpg}
\else
  \DeclareGraphicsExtensions{.eps}
\fi

\usepackage{enumitem}
\setlist[enumerate]{leftmargin=.5in}
\setlist[itemize]{leftmargin=.5in}

\newsiamremark{remark}{Remark}
\newsiamremark{hypothesis}{Hypothesis}
\crefname{hypothesis}{Hypothesis}{Hypotheses}
\newsiamthm{claim}{Claim}
\newsiamremark{example}{Example}

\headers{Escape of Strict Saddle Sets on Manifold}{Thomas Y. Hou, Zhenzhen Li, and Ziyun Zhang}

\title{Analysis of Asymptotic Escape of Strict Saddle Sets in Manifold Optimization}%\thanks{Submitted to the editors DATE.
\author{Thomas Y. Hou\thanks{Applied and Computational Mathematics, Caltech, Pasadena, CA, 91125 (\email{hou@cms.caltech.edu}, \email{zhenzhen@caltech.edu}, \email{zyzhang@caltech.edu}).}
\and Zhenzhen Li \footnotemark[2]%\thanks{Math Department, Hong Kong University of Science and Technology, Hong Kong (\email{zlice@ust.hk}).}
\and Ziyun Zhang \footnotemark[2]}%\thanks{Applied and Computational Mathematics, Caltech, Pasadena, CA, 91125 (\email{zyzhang@caltech.edu}).}}

\def\N{\mathcal{N}}
\def\R{\mathbb{R}}
\def\C{\mathbb{C}}
\def\bz{\bm{z}}
\def\bZ{\bm{Z}}
\def\bX{\bm{X}}
\def\bU{\bm{U}}
\def\bV{\bm{V}}

\def\M{\mathcal{M}}
\def\S{\mathbb{S}}
\def\grad{\nabla}

\usepackage{amsopn}

\graphicspath{{./figures/}}

\newcommand{\rev}[1]{#1}
\usepackage[normalem]{ulem}
\usepackage{soul}

\ifpdf
\hypersetup{ pdftitle={Analysis of Asymptotic Escape of Strict Saddle Sets in Manifold Optimization}, pdfauthor={Thomas Y. Hou, Zhenzhen Li, and Ziyun Zhang} }
\fi

\begin{document}

\maketitle
\begin{abstract}
   In this paper, we provide some analysis on the asymptotic escape of strict saddles in manifold optimization using the projected gradient descent (PGD) algorithm. One of our main contributions is that we extend the current analysis to include non-isolated and possibly continuous saddle sets with complicated geometry. We prove that the PGD is able to escape strict critical submanifolds under certain conditions on the geometry and the distribution of the saddle point sets. \rev{We also show that the PGD may fail to escape strict saddles under weaker assumptions even if the saddle point set has zero measure. %and there is a uniform escape direction. We provide a counterexample to illustrate this important point.
     }We apply this saddle analysis to the phase retrieval problem on the low-rank matrix manifold, prove that there are only a finite number of saddles, and \rev{that in a specific region,} they are strict saddles with high probability. We also show the potential application of our analysis for a broader range of manifold optimization problems.
\end{abstract}



\section{Introduction}
\label{sec:intro}
%\subsection{Motivation and questions}
%\zzy{Simplify the first part as a short version of the intro of paper 2, but add some necessary information about the low-rank matrix manifold}

%Non-convex optimization has remained an open problem for a long time due to its absence of global convergence to its global minimum. However, recently some non-convex optimization applications in machine learning problems are found to work surprisingly well. An even more exciting thing is that some of them are shown to have theoretical guarantee of global convergence to its global minimum. Namely, in Euclidean space, gradient descent can asymptotically avoid saddle point and converge to minimizers \cite{Jason2017}\cite{Jason2015}. In addition, in many machine learning problems, such as matrix sensing, matrix completion, phase retrieval, tensor decomposition, dictionary learning, and robust PCA, it is proved that the only local minimum is global minimum \cite{MaSen}\cite{MaCompletion}\cite{TenDecomp}\cite{PhaseRetrieval}\cite{robustPCA}\cite{DicLearn}\cite{PhaseRetrieval2}. These two discoveries imply that, with random initialization, simple gradient descent can generate convergent sequence to the global minimum.

Manifold optimization has long been studied in mathematics. From signal and imaging science \cite{ImagingOsher1}\cite{SigImg2009}, to computer vision \cite{CompVisionGrass}\cite{CompVisionRiemann} and quantum information \cite{kolodrubetz2013classifying}\cite{Quantum1}\cite{zanardi2007information}, it finds applications in various disciplines. However, it is the field of machine learning that sees the most diverse applications. Examples include matrix sensing \cite{NonconvOverview} and matrix completion \cite{ManiOpt1} on the low-rank matrix manifold $\{\bZ:\, \bZ\in\R^{n\times n},\, rank(\bZ)=r\}$; independent component analysis on the Stiefel manifold \cite{ICA1}; covariance estimation on the low-rank elliptope \cite{CovEst}; kernel learning, feature selection and dimension reduction on the Grassmannian manifold \cite{FeaKerLearn}; dictionary learning on special orthogonal group manifold \cite{DictLearn}; and blind deconvolution \cite{bd2018}\cite{shi2020manifold}, point cloud denoising \cite{ImagingOsher1}\cite{PointCloud2}, tensor completion \cite{da2015optimization}\cite{kasai2016low}\cite{kressner2014low}, metric learning \cite{MetricLearn}, Gaussian mixture \cite{GM}, just to name a few.  {Among them, low-rank matrix recovery problems have drawn the most attention in recent years, mainly because of the intrinsic low-rank structure that naturally comes in real-world applications.}

Convergence analysis of manifold optimization is similar to that of classical optimization \cite{smith1994optimization}, except that spurious local minima and saddle points might occur due to possible nonconvex structure of the manifold. However, numerous previous works on asymptotic convergence analysis reveal that many applications are actually void of spurious local minima \cite{robustPCA}\cite{MaCompletion}\cite{PhaseRetrieval2}. More importantly, they point out that the saddle points do not interfere with the convergence of various first-order algorithms towards a minimum, as long as they are ``strict saddles'', even though saddles are first-order stationary points \cite{TenDecomp}\cite{Jinchi2017}. Intuitively speaking, a strict saddle point is a stationary point whose tangent space has a direction with negative curvature. The point series generated by the algorithm tend to ``slip away'', or escape, from strict saddles because of this negative curvature direction. This ``escape and convergence'' property is observed not only for stochastic algorithms, but also for deterministic ones, such as the simplest fixed-stepsize gradient descent.  

The current analysis on strict saddles mostly focuses on isolated saddle points \rev{in the Euclidean space \cite{Jason2017}\cite{Jason2015}. There is a lack of systematic analysis on the Riemannian manifolds, and even less for the explicit treatment for non-isolated continuous saddle sets, despite their prevalence.} %There is a lack of systematic analysis for non-isolated (continuous) saddle sets on the Riemannian manifolds. Explanation of asymptotic escape abounds in the Euclidean spaces \cite{Jason2015}\cite{Jason2017}, but not on the Riemannian manifolds. Even less is devoted to the explicit treatment of the whole critical point set, despite the fact that its geometry can be quite complicated. 
One example is the variational linear eigenvalue problem on the sphere $\mathbb{S}^{n-1}$. It can have infinitely many saddle points if a non-minimal eigenvalue is duplicated. Another is the Gaussian phase retrieval problem, where the critical points are even more elusive due to the stochasticity of the model. This motivates us to provide some general analytic tools to study these algorithms.

One of the main contributions of this paper is that we provide a systematic analysis for the asymptotic escape of non-isolated and possibly continuous saddle sets with complicated geometry. Based on the analytic tools that we develop to study optimization on low-rank manifolds, we prove that the PGD is able to escape strict critical submanifolds under certain conditions. These conditions are concerned with some geometric property of the saddle point set or the distribution of these saddle points. We argue that these conditions are necessary to guarantee the asymptotic escape of the strict saddles by the PGD in manifold optimization. \rev{However, these conditions are not stringent and are usually satisfied by familiar applications. We compare our conditions with those of the recent work \cite{panageas2016gradient}, and point out that these two are consistent. We also give some examples that violate the conditions and result in failures of asymptotic escape, for the purpose of theoretical interest. 

%To support this argument, we construct a counterexample (Example \ref{ex:counter}), which shows that the PGD fails to escape strict saddles under weaker assumptions even if the strict saddle point set has zero measure. In this counterexample, the saddles are all strict saddles. They have zero measure on the manifold and have a common escape direction. Yet there is a positive probability (16.67\% probability in our example, and can be even higher) that the PGD with a random initialization fails to avoid these strict saddles. The reason is that these saddles points do not form countable connected submanifolds or have any favorable geometric structure. This example gives us a healthy warning that one cannot claim that strict saddles are ``harmless'' before we analyze carefully the geometric structure and the distribution of the strict saddle point set.
}

What lies at the core of this asymptotic analysis is an interesting interplay of dynamical systems and nonconvex optimization, and a translation of languages from the Morse theory \cite{banyaga2010morse}\cite{banyaga2013lectures}\cite{cohen1991topics} into gradient flow lines and further into gradient descents. Although these tools were initially developed to study homology, they have provided invaluable insight into the converging/escaping sets of strict saddle points with nontrivial geometry. We draw inspirations from them and propose a new unified tool to analyze asymptotic convergence and escape from saddle points.

\rev{We are aware that there is a parallel line of research on the stochastic/perturbed version of gradient descent, as well as its variant on the Riemannian manifold. Recent works including \cite{criscitiello2019efficiently}\cite{TenDecomp}\cite{Jinchi2017}\cite{sun2019escaping} show that the stochastic/perturbed gradient descent is a powerful tool to get rid of saddles and does not impose any constraint on the geometry of saddle sets as we do. The reason our analysis focuses on the unperturbed gradient descent, is that only by eliminating the perturbation effect can we single out the essential property of gradient descent itself. The development of a thorough asymptotic theory for this simple PGD algorithm is crucial towards the understanding of why simple PGD works sufficiently well in many applications. }

As an application of our asymptotic escape analysis for strict saddles, \rev{we consider the phase retrieval problem \cite{fienup1982phase}\cite{gerchberg1972practical}\cite{jaganathan2015phase}\cite{shechtman2015phase} that has received considerable attention in the recent years.} \rev{We combine the perspectives of Riemannian manifold optimization \cite{cai2018solving} and landscape analysis \cite{PhaseRetrieval2}\cite{PhaseRetrieval} and derive new results.}
%\cite{cai2018solving} solves the phase retrieval problem by the Riemannian manifold optimization, \cite{PhaseRetrieval} analyses the landscape of phase retrieval problem under the Euclidean setting, and \cite{PhaseRetrieval2} use the truncation technique to get favourable landscape with optimal sampling complexity.} 
We analyze the saddle points of the phase retrieval problem on the low-rank matrix manifold. Surprisingly, we are able to prove that there are only finite number of saddles and they are all strict saddles with very high probability. Our analysis provides a rigorous explanation for the robust performance of the PGD in solving the phase retrieval problem as a low-rank optimization problem, a phenomenon that has been observed in previous applications reported in the literature \rev{\cite{PhaseRetrieval2}}. 

Although our primary focus is the low-rank matrix manifold, the asymptotic convergence to the minimum and the escape of strict saddles (strict critical submanifolds) are also valid on any arbitrary finite dimensional Riemannian manifold. In particular, the properties of the PGD are well preserved if the manifold is embedded in a Banach space and inherits its metric.  {Popular examples of manifold optimization include optimization problems on the sphere, the Stiefel manifold, the Grassmann manifold, the Wasserstein statistical manifold \cite{li2018natural}, the flag manifold for multiscale analysis \cite{ye2019flag}, and the low-rank tensor manifold \cite{holtz2012manifolds}\cite{lubich2015time}. They find applications in many fields including physics, statistics, quantum information and machine learning.}

To illustrate this, we consider the optimization on the unit sphere and the Stiefel manifold as two examples of applications. In the first example, we consider a variational eigenvalue problem on a sphere. Both linear and nonlinear eigenvalue problems are considered. In the case of linear eigenvalue problem, we show that the first eigenvector is the unique global minimum with a positive Hessian, and all the subsequent eigenvectors are all strict saddles whose Hessian has at least one negative curvature direction given by the eigenvector itself. Thus, our asymptotic escape analysis applies to this manifold optimization problem. In the case of nonlinear eigenvalue problem, we cannot guarantee by our analysis that the saddles are all strict saddles. But we can verify this numerically. Our numerical results show that the PGD algorithm gives good convergence for both linear and nonlinear eigenvalue problems. In the case when we have a cluster of eigenvalues, we further propose an acceleration method by formulating a simultaneous eigen-solver on the Stiefel manifold. We observe a considerable speedup in the convergence of the PGD method on the Stiefel manifold.

%Apart from the low-rank matrix recovery problem, there are many other popular manifold optimization problems. In particular, many of the manifolds of interest share a similar geometry with low-rank matrix manifold, e.g. the unit sphere, the Stiefel manifold and the well-studied Grassmann manifold. Namely, they are embedded in some Banach space, inherit its metric, and on which a natural high order retraction operator can be defined. The projected gradient descent (PGD) method is thus also applicable as a handy version of Riemannian gradient descent on these manifolds. 

\rev{We point out that the analysis of this paper is purely asymptotic. The quantitative study of convergence rate will be the focus of our future work \cite{paper2}. Notably, in contrast to the caveat that gradient descent could take exponential time to escape saddles in the worst-case scenario \cite{du2017gradient}, empirical evidence in the figures of Sections \ref{sec:pr} to \ref{sec:app} and many examples reported in the literature demonstrate that we usually have almost linear convergence to minimizers.}

The rest of the paper is organized as follows. Section \ref{sec:PGD} contains the main results on asymptotic escape of strict saddles (in particular non-isolated ones) on the Riemannian manifolds. In Section \ref{sec:Mr}, we explore the geometric structure of the closed low-rank matrix manifold $\overline{\M_r}$. In Section \ref{sec:pr}, phase retrieval is analyzed as an example of asymptotic escape of saddles on $\overline{\M_r}$. We extend the application to other manifolds and a broader scope of problems in Section \ref{sec:app}. Finally, we we make some concluding remarks and discuss future works in Section \ref{sec:conclusion}.

\section{The projected gradient descent (PGD) and asymptotic escape of non-isolated strict saddles}
\label{sec:PGD}

In this section, we discuss the optimization technique we use, namely the projected gradient descent (PGD) with retraction onto the manifold. We will prove an important property of the proposed technique, namely it is able to escape any strict saddle points and always converge to minimizers. This may be obvious in the Euclidean spaces but not so obvious on manifolds. %Indeed, it plays an important role in ensuring good performance of PGD on manifolds when combined with the favorable landscapes of specific problems.

We stress that although our motivations (and hence notation conventions) are based on the low-rank matrix manifold, what we discuss in this section can be easily generalized to arbitrary finite dimensional Riemannian manifolds as is mentioned in Section \ref{sec:intro}. This works as long as either (1) $\M$ is embedded in an ambient Banach space and inherits its metric (so that the \emph{embedded} gradient $\nabla$ is well-defined), and there exists a well-defined first-order retraction $R$ ; or (2) $\M$ is ``flat'', so that no retraction $R$ is needed. 

The asymptotic escape of strict saddle points combined with a good landscape of the objective function can lead to asymptotic convergence to minimizers. This could serve as a fundamental tool for various application tasks. More applications will be discussed Section \ref{sec:app}.

\subsection{Projected gradient descent on the manifold}
Assume we are given a function $f(\cdot):\M\rightarrow \mathbf{R}$ where $\M$ can be a general manifold. We start from a proper initial guess $\bZ_0\in\M$. The iteration points $\{\bZ_n\}_{n=0}^N$ are generated by
\begin{equation}
\label{eq:pgd}
    \bZ_{n+1}=\mathcal{R}\left(\bZ_n - \alpha_n P_{T_{Z_n}}(\grad f(\bZ_n))\right).
\end{equation}
Here $\nabla f$ is the \emph{embedded} gradient of $f$ in its ambient Banach space , $P_{T_{Z_n}}$ is the projection onto the tangent space of $\M$ at point $\bZ_n$, $\alpha_n$ is the $n$-th step size, and $\mathcal{R}: T_Z \rightarrow \M$ is a first-order retraction. The retraction operation is necessary in that it makes sure the generated iteration point still stays on the manifold $\M$. Specifically, we define the retraction as follows:

\begin{definition}[{Retraction}] %[\bf{Retraction}] 
Let $\|\cdot\|$ be the norm of the embedded Banach space of $\M$. Let $T_{\bZ}$ be the tangent space (or tangent cone) of $\M$ at $\bZ$. We call $\mathcal{R}_{\bZ}:T_{\bZ}\rightarrow\M$ a \emph{retraction}, if for any $\xi\in T_{\bZ}$,
    \begin{displaymath}\label{eq:first_order_retraction}
        \lim_{\alpha\rightarrow 0^+}\frac{\|\mathcal{R}_{\bZ}(\alpha\xi)-(\bZ+\alpha\xi)\|}{\alpha}=0.
    \end{displaymath}
    We also write $\mathcal{R}(\bZ+\alpha\xi)$ equivalently for $\mathcal{R}_{\bZ}(\alpha\xi)$.
\end{definition}
We will refer to (\ref{eq:first_order_retraction}) as the {\emph{first-order retraction property}}.

\begin{remark}
It is worth mentioning that there exist other manifold optimization techniques, e.g. without the projection step $T_{\bZ_n}$. We choose the current projected gradient descent algorithm here mainly for two reasons:
\begin{enumerate}[label={(\arabic*)}]
    \item Most of the operations we list here rely heavily on tangent bundles. For example, the first-order retraction (\ref{eq:first_order_retraction}) only holds for $\xi\in T_{\bZ}$. The embedded gradient $\nabla f$ is not always in $T_{\bZ}$, only $P_{T_{\bZ}}(\grad f(\bZ))$ is.
    In fact, only the projected embedded gradient corresponds to the true \emph{Riemannian} gradient: $P_{T_{\bZ}}(\grad f(\bZ)) = \text{grad} f(\bZ)$, where $\text{grad} f(\bZ)$ is the Riemannian gradient of $f$ on $\M$ at point $\bZ$.
    \item Retraction $\mathcal{R}_{\bZ}$ can be computed more efficiently when the increment lies in the tangent space. For example on the low-rank matrix manifold it involves a smaller-scale SVD \cite{Reinhold}\cite{ManiOpt1}\cite{ManiOpt2}. 
\end{enumerate}
\end{remark}

%This second-order retraction always exists due to the nice curvature of $\M$. To see it is indeed true in low-rank manifold, let us consider any point in neighborhood of $\bX=\bU\bU^\top$, denoted as $\bZ=\bX+\alpha\tilde{\xi}=(\bU+\alpha\Delta)(\bU+\alpha\Delta)^\top=\bU\bU^\top+\alpha(\bU\Delta^\top+\Delta\bU^\top)+\alpha^2\Delta\Delta^\top$, denote $\xi$: $\mathcal{R}(\bX+\alpha\xi)=\bX+\alpha\tilde{\xi}$. Since $\bU\bU^\top+\alpha(\bU\Delta^\top+\Delta\bU^\top)\in T_X$, so $\alpha\xi=\alpha(\bU\Delta^\top+\Delta\bU^\top)+\mathcal{O}(\alpha^2)$. Therefore, 
%\begin{align*}
%    \lim_{\alpha\rightarrow 0^+}\frac{\|\mathcal{R}(\bZ+\alpha\xi)-(\bZ+\alpha\xi)\|_F}{\alpha}=\lim_{\alpha\rightarrow 0^+}\frac{\mathcal{O}(\alpha^2)}{\alpha}=0
%\end{align*}
%That is $\xi\in T_Z: \mathcal{R}(Z+\alpha\xi)=Z+\alpha\tilde{\xi}=Z+\alpha\xi+o(\alpha\|\xi\|_F)$, by \cite{FirstRetract}\cite{Reinhold}\cite{ManiOpt1} the existence of such retraction holds.

\subsection{Asymptotic escape of isolated saddle points}

Consider the projected gradient descent (PGD) method with a fixed step size $\alpha$. Let $\varphi$ be the iteration operation, and
\begin{equation*}
    \bZ_{n+1} = \varphi(\bZ_n) := R(\bZ_n - \alpha P_{T_{\bZ_n}} (\nabla f(\bZ_n))).
\end{equation*}

Let $\bX^*$ be a critical point of $f$. The most common non-minimizer critical point is a ``strict saddle''. Intuitively speaking, a strict saddle point $\bZ^*$ is a point around which there is a strictly decreasing direction ($\text{Hess}f(\bZ^*)$ has a negative eigenvalue) and no flat direction ($\text{Hess}f(\bZ^*)$ has no zero eigenvalue). To rigorously define strict saddle points we need the following definitions:

\begin{definition}[Levi-Civita connection]
    \rev{The \emph{Levi-Civita connection} $\widetilde{\nabla}_{\xi}\eta$, acting on two vectors or vector fields $\eta$, $\xi$ in the tangent bundle $T\M$, is the unique affine connection on $\M$ that preserves the metric and is torsion-free.}
\end{definition}
Note that this is not to be confused with the operator $\nabla$ in (\ref{eq:pgd}), which is used to denote the gradient in the ambient space. 

\begin{definition}[Riemannian gradient]
    Given $f: \M \rightarrow \R$, the \emph{Riemannian gradient} of $f$ is the vector field $\text{grad}f$, such that for any vector field $Y$ on $\M$,
    \begin{align*}
        \langle \text{grad}f, Y\rangle = Y(f),
    \end{align*}
    where $\langle \cdot, \cdot \rangle$ is the metric on $\M$ and $Y(\cdot)$ is the vector field action, i.e. $Y(f) = \sum_{i}Y_i\frac{\partial f}{\partial E_i}$ for a basis $\{E_i\}$.
\end{definition}

\begin{remark}
    The Riemannian gradient is equivalent to the tangent space projection of the embedded gradient in the ambient space, i.e.
    \begin{align*}
        \text{grad } f(\bZ) &= P_{T_{\bZ}}(\nabla f(\bZ)).
    \end{align*}
    Further, if the metric of $\M$ is inherited from the ambient space, then the Levi-Civita connection on $\M$ is the tangent space projection of the Levi-Civita connection (natural gradient) of the ambient space. \rev{Specifically, for $\eta$, $\xi \in T\M$, we have}
    \begin{align*}
        \widetilde{\nabla}_{\xi} \eta = P_{T_{\bZ}}(\nabla \eta [\xi]).
    \end{align*}
\end{remark}

\begin{definition}[Hessian]
\label{def:hessian}
    Given a function $f: \M \rightarrow \R$, the \emph{{Riemannian Hessian}} of $f$ at point $\bZ$ is $\text{Hess } f(\bZ): T_{\bZ}\M \rightarrow T_{\bZ}\M$ defined by
    \begin{equation}
        \text{Hess } f(\bZ) [\xi] = \widetilde{\nabla}_{\xi}\text{grad }f(\bZ).
    \end{equation}
    where $\widetilde{\nabla}_{(\cdot)}(\cdot)$ is the Levi-Civita connection on $\M$.
\end{definition}

\begin{proposition}
    If the retraction is second-order, i.e.
    \begin{align*}
        P_{T_{\bZ}}\left(\frac{d^2}{d\alpha^2}\mathcal{R}_{\bZ}(\alpha \xi)\mid_{\alpha=0}\right) = 0,
    \end{align*}
    then 
    \begin{equation}
    \label{eq:LC_alternative}
        \text{Hess } f(\bZ) = \text{Hess } (f\circ \mathcal{R}_{\bZ})(0).
    \end{equation}
\end{proposition}

%\begin{definition}[Hessian]
%\begin{remark}
    In particular, (\ref{eq:LC_alternative}) is true for the low-rank matrix manifold, and this gives a more computable definition of Hessian, see also \cite{ManiOpt1}.
    It is proved in \cite{absil2013extrinsic} that in the case of the low-rank matrix manifold, the above expression recovers Definition \ref{def:hessian}. 
    %However, this is not true on general manifolds with nontrivial Levi-Civita connections, and we must turn to the original definition.
%\end{remark}

\begin{definition}[Strict saddle point]%, the narrow definition] 
\label{def:saddle1}
We call $\bZ\in\M$ a \emph{strict saddle point} of $f(\cdot):\M\rightarrow\R$, if 
    \begin{enumerate}
        \item $P_{T_{\bZ}}(\grad f(\bZ))=0$;
        \item $\text{Hess }f({\bZ})$ as a linear operator has at least one positive eigenvalue, at least one negative eigenvalue and no zero eigenvalue. 
    \end{enumerate}
\end{definition}

\rev{
In contrast to Definition \ref{def:saddle1}, we call a point $\bZ\in\M$ a \emph{local minimizer} if $\text{Hess }f({\bZ})$ is positive semi-definite. Note that this is only legitimate in the broader sense, as the PGD algorithm does not distinguish between local minima and degenerate saddles, i.e. saddles that only have higher than second order negative curvature.
}

Then we have the following main result.

\begin{theorem}[{PGD asymptotically only converges to a local minimum}]%[\bf{PGD asymptotically only converges to local minimum}] 
\label{thm:escape}
    \rev{Let $f(\cdot):\M\rightarrow\R$ be a $C^2$ function on $\M$.} Suppose that $f$ has either finitely many saddle points (denoted as set $S$), or countably many saddle points in a compact submanifold of $\M$, and all saddle points of $f$ are strict saddles as is defined in Definition \ref{def:saddle1}. Let $C$ be set of all local minimizers, $\{\bZ_n\}$ be the series of points generated by the projected gradient descent algorithm on $\M$, then we have:
    \begin{enumerate}
        \item $Pr(\lim_{n\rightarrow\infty} \bZ_n\in S)=0$;
        \item If $\lim_{n\rightarrow\infty} \bZ_n$ exists, then $Pr(\lim_{n\rightarrow\infty} \bZ_n\in C)=1$.
    \end{enumerate}    
\end{theorem}

In other words, the PGD with a random initialization is unlikely to converge to a saddle point $\bZ^*$ as long as $\bZ^*$ is a strict saddle. 
%It should be pointed out that although such result is already proved for gradient descent in Euclidean spaces \cite{Jason2015}, it has never been proved on non-convex manifolds to the best of our knowledge. This result and the idea behind its proof can be very useful in the asymptotic analysis of optimization on manifolds.

To prove this result, the main tool is the stable manifold theorem on low-rank matrix manifolds, which is an extension of similar results in the Euclidean spaces.

\begin{definition}[Definition 4.13 in \cite{banyaga2013lectures}]
    A fixed point $p\in \M$ of a smooth diffeomorphism $\varphi: \M \to \M$ is called \emph{hyperbolic} iff none of the complex eigenvalues of $D\varphi_p: T_p\M \to T_p\M$ have length 1, where $D\varphi$ is the Riemannian gradient of $\varphi$ at point $p$.
\end{definition}

For a fixed point $p$ of $\varphi$, there is a splitting of $T_p\M$ that is preserved by $\varphi$:
\begin{equation}
\label{eq:split1}
    D\varphi_p: T^s_p\M \oplus T^u_p\M \to T^s_p\M \oplus T^u_p\M,
\end{equation}
where $\varphi$ is contracting on $T^s_p\M$ and expanding on $T^u_p\M$. Since the manifold $\M$ is finite dimensional, for a hyperbolic fixed point $p$ on $\M$, we can always find a $\lambda \in (0,1)$ such that 
\begin{equation*}
    \|D\varphi|_{T^s_p\M}\| < \lambda, \quad \|(D\varphi|_{T^u_p\M})^{-1}\| < \lambda.
\end{equation*}  

\begin{theorem}[Theorem 4.15 in \cite{banyaga2013lectures}]
\label{thm:stablemanifold1}
    If $\varphi: \M\to \M$ is a smooth diffeomorphism of a finite dimensional smooth manifold $\M$, and $p$ is a hyperbolic fixed point of $\varphi$, then
    \begin{equation*}
        W^s_p(\varphi) := \{x \in \M | \lim_{n \to \infty} \varphi^n(x) = p \}
    \end{equation*}
    is an immersed submanifold of $\M$ with $T_pW^s_p(\varphi) = T^s_p\M$. Moreover, $W^s_p(\varphi)$ is the surjective image of a smooth injective immersion
    \begin{equation*}
        E^s: T^s_p\M \to W^s_p(\varphi) \subseteq \M.
    \end{equation*}
    Hence, $W^s_p(\varphi)$ is a smooth injectively immersed open disk in $\M$.
    We call it the \emph{stable manifold} of $p$ with respect to $\varphi$. The \emph{unstable manifold} $W_p^u(\varphi)$ of $p$ can also be defined accordingly.
\end{theorem}

\begin{proof}[{Proof of Theorem \ref{thm:escape}}] %[\bf{Proof of Theorem \ref{thm:escape}}] 
From Theorem \ref{thm:stablemanifold1}, if $\varphi$ is a diffeomorphism, and $\bZ^*$ is a hyperbolic point of $\varphi$, then $W^S_{\bZ^*}(\varphi)$, the stable set of $\bZ^*$, will be a lower dimensional submanifold in $\overline{\M}$. Then, the converging set of $\bZ^*$ will have measure 0 with respect to the manifold, and any random initialization of PGD will escape such a strict saddle point almost surely. 

We first prove that $\bZ^*$ is a hyperbolic point. The diffeomorphic property of $\varphi$ is actually contained in the proof of the former.

Given $\xi \in T_{\bZ^*}\M$, for any $\tilde{\xi}$ that satisfies $\bZ^* + \tilde{\xi} \in \M$, $P_{T_{\bZ^*}}(\tilde{\xi}) = \xi$, we have
\begin{align*}
    \text{Hess}f(\bZ^*)[\xi] + \mathcal{O}(\|\xi\|^2) &= \widetilde{\nabla}_\xi \text{ grad} f(\bZ^*) \\
    &= P_{T_{\bZ^*}} (\nabla \text{grad} f(\bZ^*)[\xi]) \\
    &= P_{T_{\bZ^*}} (\text{grad} f(\bZ^*+\tilde{\xi})-\text{grad} f(\bZ^*)) + \mathcal{O}(\|\xi\|^2) \\
    &= P_{T_{\bZ^*}} (\text{grad}f (\bZ^*+\tilde{\xi})) +  \mathcal{O}(\|\xi\|^2).
\end{align*}
Note that $\text{grad} f(\bZ^*) = 0$ since $Z^*$ is a critical point. Therefore, for $ \varphi(\bZ_n) = R(\bZ_n - \alpha P_{T_{\bZ_n}} (\nabla f(\bZ_n)))$, 
\begin{align*}
    D\varphi_{\bZ^*} [\xi] 
%    &= \nabla_{\xi} \varphi(\bZ^*)  \\
    &= P_{T_{\bZ^*}}(\varphi(\bZ^*+\tilde{\xi})-\varphi(\bZ^*)) + \text{o}(\|\xi\|) \\
    &= P_{T_{\bZ^*}} (R(\bZ^*+\tilde{\xi}-\alpha \text{ grad} f(\bZ^*+\tilde{\xi})) -\bZ^*) + \text{o}(\|\xi\|) \\
    &= P_{T_{\bZ^*}}(\bZ^* + \tilde{\xi}-\alpha \text{ grad} f(\bZ^*+\tilde{\xi}) + \text{o}(\|\xi\|) - \bZ^*) \\
    &= P_{T_{\bZ^*}}(\tilde{\xi}-\alpha \text{ grad} f(\bZ^*+\tilde{\xi})) + \text{o}(\|\xi\|) \\
    &= \xi-\alpha \text{ Hess} f(\bZ^*)[\xi] + \text{o}(\|\xi\|).
\end{align*}
We have
\begin{align*}
    D\varphi_{\bZ^*} [\xi] = \xi - \alpha\text{Hess}(f)(\bZ^*)[\xi],
\end{align*}
i.e.
\begin{equation*}
    D\varphi_{\bZ^*} = I - \alpha \text{ Hess}(f)(\bZ^*).
\end{equation*}
Thus $\bZ^*$ being strict saddle implies that, by choosing $\alpha$ sufficiently small (but only depending on the eigenvalues of $\text{Hess}(f)(\bZ^*)$), $\bZ^*$ is hyperbolic with respect to $\varphi$.

Now, $\varphi$ is a diffeomorphism at $\bZ^*$ because, if we choose $\alpha$ small enough so that $\|\text{Hess}(f)(\bZ^*)\|<\frac{1}{\alpha}$, then $D\varphi$ is always invertible and bounded. If there are only finitely many strict saddle points, or there are countably infinite number of them in a compact region, $\|\text{Hess}(f)(\bZ^*)\|$ shall be upper bounded, and such an $\alpha$ is always attainable.

Using Theorem \ref{thm:stablemanifold1}, the set of points on $\M$ that converge to $\bZ^*$ is a lower dimensional submanifold in $\M$, which has measure 0. We could safely deduce that, by randomly sampling a start point $\bZ_0$ in $\M$, the probability of converging to a strict saddle point is 0, i.e.
\begin{equation*}
    \text{Prob}(\lim_{k\to\infty}\bZ_k = \bZ^*) = 0.
\end{equation*}
Since there are only countably many strict saddle points, $\cup_{\bZ^*\in S}W^S_{\bZ^*}(\varphi)$ still has measure 0. So the PGD with a randomly sampled starting point converges to any point in $S$ with probability 0. \rev{This proves the first argument. 

As for the second argument, since the step size is a constant $\alpha$, the only stationary points of the algorithm are first-order critical points of the loss function. The local maximizers are ruled out by the descent property of gradient descent. So if the limit point exists, it is a local minimizer with probability 1.}
\end{proof}

\begin{remark}
    As is mentioned at the beginning of this section, another case is when the manifold $\M$ is ``flat'' and no retraction is needed, i.e. $\bZ+\xi \in \M$ for any $\bZ \in \M$ and $\xi \in T_{\bZ}\M$. If no embedding Banach space is present, the optimization technique might be reduced to using the Riemannian gradient directly. The above proof is also reduced to applying $D$ directly to $\bZ-\text{grad } f(\bZ)$ and is trivially true.
\end{remark}

%\hspace{8pt}

The natural question to ask is: what if $\text{Hess}(f)(X^*)$ has zero eigenvalue(s) in addition to negative eigenvalues? In this case, $D\varphi$ has \emph{center} directions that are neither expanding nor contracting. This can be taken into account by extending the stable manifold theorem to a slightly stronger version, which is the following. 

\begin{theorem}%[Extension of theorem III.7 in \cite{shub2013global}]
\label{thm:split_extension}
    Let $\varphi: \M\to \M$ be a smooth diffeomorphism of a finite dimensional smooth manifold $\M$, and $p$ is a fixed point of $\varphi$. Assume that 
    \begin{equation}
    \label{eq:split2}
        T_p\M = T_p^s\M \oplus T_p^c\M \oplus T_p^u\M
    \end{equation}
    which is the invariant splitting of $T_p\M$ into contracting, centering, and expanding subspaces corresponding to eigenvalues of magnitude less than, equal to, and greater than 1. Let
    \begin{equation*}
        T_p^{cs}\M := T_p^s\M \oplus T_p^c\M.
    \end{equation*}
    Then we have
    \begin{equation*}
        W^s_p(\varphi) := \{x \in \M | \lim_{n \to \infty} \varphi^n(x) = p \}
    \end{equation*}
    is an immersed submanifold of $\M$ and $T_pW^s_p(\varphi) \subseteq T_p^{cs}\M$. We call it the \emph{(generalized) stable manifold} of $p$ with respect to $\varphi$. 
\end{theorem}

For those who are interested in the proof, a detailed one for the Euclidean case can be found in Theorem III.7 in \cite{shub2013global}, and the extension to the manifold is similar to Theorem \ref{thm:stablemanifold1}.

%\begin{remark}
%    The original theorem in \cite{shub2013global} only states the result in finite dimensional Euclidean space. But since Riemannian manifold is locally analogous to Euclidean space, all the results can be painlessly extended to finite dimensional Riemannian manifolds (similar as \cite{banyaga2013lectures}) except for injective immersion.
%\end{remark} 

\begin{definition}[Strict saddle point, the general definition] 
\label{def:saddle2}
We call $\bZ\in\M$ a \emph{strict saddle point} to $f(\cdot):\M\rightarrow\R$, if 
    \begin{enumerate}
        \item $P_{T_{\bZ}}(\grad f(\bZ))=0$;
        \item $\text{Hess }f({\bZ})$ has at least one negative eigenvalue. 
    \end{enumerate}
\end{definition}

Thus, for a strict saddle point $\bZ^*$ where $\text{Hess}(f)$ possibly has zero eigenvalue(s), we still have $\text{dim}(W^s_p(\varphi)) \le \text{dim}(T_p^{cs}\M) < \text{dim }\M$, and we have $\text{Prob}(\lim_{k\to\infty}\bZ_k = \bZ^*) = 0$. Therefore, we have the following theorem.

\begin{theorem} 
\label{thm:escape2}
    \rev{Let $f(\cdot):\M\rightarrow\R$ be a $C^2$ function on $\M$.} Suppose that $f(\cdot):\M\rightarrow\R$ has either finitely many saddle points, or countably many saddle points in a compact submanifold of $\M$, and all saddle points of $f$ are strict saddles as is defined in Definition \ref{def:saddle2}. Then the results of Theorem \ref{thm:escape} still holds, i.e.
    \begin{enumerate}
        \item $Pr(\lim_{n\rightarrow\infty} \bZ_n\in S)=0$;
        \item If $\lim_{n\rightarrow\infty} \bZ_n$ exists, then $Pr(\lim_{n\rightarrow\infty} \bZ_n\in C)=1$.
    \end{enumerate}
\end{theorem}

%\subsection{Extension to general finite-dimensional manifolds}

\subsection{A closer look at the critical points: non-isolated case}
\label{sec:submanifold}
%The above results on stable and unstable manifolds for isolated critical points are widely used in Morse theory, where the so-called \emph{Morse function} is the function whose critical points all have non-degenerate Hessian \cite{banyaga2013lectures}. Generalizations of the Morse function include the \emph{Morse-Smale function}, whose stable and unstable manifolds all intersect transversally, and the \emph{Morse-Bott function}, whose critical points form finitely many submanifolds and have non-degenerate normal Hessian \cite{banyaga2010morse}\cite{cohen1991topics}. 

%Although these tools were initially developed to study homology, we stress that they have provided invaluable insight into the converging/escaping sets of strict saddle points. For example, there are functions with infinite strict saddle points that form a submanifold (e.g. see Section \ref{sec:app}), but whose converging set should reasonably have measure 0. Empirical evidence shows satisfactory convergence of PGD to its minimum, but why? Saddle escaping theory for this situation has still been lacking. So we now develop a tool to this purpose.

As is mentioned in Section \ref{sec:intro}, it is very common that there are more than a countable number of strict saddle points, e.g. when they form a submanifold, or a more complicated set, with Lebesgue measure 0. Empirical evidences show satisfactory convergence of the PGD to its minimum, which indicates successful escape from these strict saddles. But there is a lack of theoretical analysis to confirm this observation. In the following, we will use some further results from the Morse theory and its extensions to provide an analytical tool for this purpose.

\begin{definition}[Critical submanifold]
   %(the infinite set of critical points that form a submanifold)
   For $f: \M\mapsto\R$, a connected submanifold $\mathcal{N} \subset \M$ is called a \emph{critical submanifold} of $f$ if every point $\bZ$ in $\mathcal{N}$ is a critical point of $f$, i.e. $\text{grad } f(\bZ) = 0$ for any $\bZ\in\N$.
\end{definition}

\begin{definition}[Strict critical submanifold]
    %as an analogy to strict saddles
    A critical submanifold $\mathcal{N}$ of $f$ is called a \emph{strict critical submanifold}, if $\forall \bZ\in \mathcal{N}$,
    \begin{align*}
        \lambda_{\text{min}}(\text{Hess } f(\bZ)) \le c < 0,
    \end{align*}
    where $\lambda_{\text{min}}(\cdot)$ takes the smallest eigenvalue, and $c = c(\mathcal{N})$ is a uniform constant for all $\bZ\in\mathcal{N}$ depending only on $\mathcal{N}$.
\end{definition}

Analogous to the stable/unstable manifolds of critical points in Theorems \ref{thm:stablemanifold1} and \ref{thm:split_extension}, we may define stable/unstable manifolds of critical submanifolds.\footnote{
    The reader shall be careful while distinguishing different ``manifolds'': the domain of the function $f$ is a manifold, and the critical set of $f$ is now a submanifold, but the names of stable/unstable manifolds are given regardless of the domain of $f$.
}

\begin{definition}[Generalized stable/unstable manifold]
\label{def:saddle3}
    %((1) the stable/unstable manifold of a certain critical submanifold; 
    %(2) the stable/unstable manifold of all the critical submanifolds: how to take their union? needs transversality.)
    Let $\varphi: \M\to \M$ be a smooth diffeomorphism of $\M$. Then for a submanifold of $\mathcal{N}\subset \M$, the \emph{stable manifold} and \emph{unstable manifold} of $\mathcal{N}$ with respect to $\varphi$ are defined as
    \begin{align*}
        W^s_\mathcal{N}(\varphi) & := \{x \in \M | \lim_{n \to \infty} \varphi^n(x) \in \mathcal{N}\}, \\
        W^u_\mathcal{N}(\varphi) & := \{x \in \M | \lim_{n \to -\infty} \varphi^n(x) \in \mathcal{N}\}.
    \end{align*}
\end{definition}

%Given a strict critical submanifold $\mathcal{N}$ of $f$, it is then natural to split the tangent space as in equation (\ref{eq:split2}). There would always be a \emph{centering} subspace $T^c_p\M$ for any $p\in\mathcal{N}$, and it always contains $T_p\mathcal{N}$, the tangent space of $\mathcal{N}$ embedded in $T_p\M$. This can be seen from the first-order vanishing condition of critical submanifold. Similar analysis as in the proof of Theorem \ref{thm:escape} would work well when there are only finite number of critical submanifolds.

Given a nontrivial strict critical submanifold $\N$ of $f$, at every point $p\in\N$, the tangent space is split as 
\begin{align*}
    T_p\M = T_p\N \oplus \nu_p\N,
\end{align*}
where $\nu_p\N$ is the normal space of $\N$ at $p$ immersed in $\M$. Similar to equation (\ref{eq:split2}), it is further split into 
\begin{align*}
    T_p\M = T_p\N \oplus (\nu_p^s\M \oplus \nu_p^c\M \oplus \nu_p^u\M).
\end{align*}
To arrive at a result similar to that stated in Theorem \ref{thm:split_extension}, it suffices to define
\begin{align*}
    T_p^{cs}\M := T_p\N \oplus (\nu_p^s\M \oplus \nu_p^c\M)
\end{align*}
and notice that $T_pW_\N^s(\varphi) \subseteq T_p^{cs}\M$ for any $p\in\N$. Since $T_p^{cs}\M$ is dimension deficient by the definition of strict critical submanifold, any random initialization still falls into the converging set of $\N$ with probability 0. Because the union of a finite number of 0-measure sets still has measure 0, the above result works well with countably many critical submanifolds. To sum up, we have the following theorem. 

\begin{theorem} 
\label{thm:escape3}
    \rev{Let $f(\cdot):\M\rightarrow\R$ be a $C^2$ function on $\M$.} Suppose that the saddle set $S$ of $f(\cdot):\M\rightarrow\R$ consists of finitely many critical submanifolds, or countably many critical submanifolds in a compact region of $\M$, and all of them are strict critical submanifolds as defined in Definition \ref{def:saddle3}. Then the results of Theorem \ref{thm:escape} still holds, i.e.
    \begin{enumerate}
        \item $Pr(\lim_{n\rightarrow\infty} \bZ_n\in S)=0$;
        \item If $\lim_{n\rightarrow\infty} \bZ_n$ exists, then $Pr(\lim_{n\rightarrow\infty} \bZ_n\in C)=1$.
    \end{enumerate}
\end{theorem}

%\vspace{6pt} 

\rev{

For situations more complicated than those stated in the above theorem, it is conjectured that the transversality relationship of submanifolds can be exploited to find out the succession relationship of critical sets. We refer the reader to the Appendix for some useful tools and interesting insights in this direction.

\vspace{6pt} 

Finally, we point out that the number of critical submanifolds being countable is an essential condition, but not a binding one. Of course, one reason of this statement is that it is often satisfied in practice. Namely, in the known applications with very complicated saddle geometries, e.g. matrix factorization and nonlinear eigenproblems, the saddles can still be grouped into countably many points or submanifolds. In those cases, either Theorem \ref{thm:escape}, Theorem \ref{thm:escape2} or Theorem \ref{thm:escape3} is applicable. 

But even from a purely theoretical point of view, the number of strict critical submanifolds being uncountable is unlikely to happen. This is in accordance with the result of a recent work \cite{panageas2016gradient}. The result explicitly includes the case of ``uncountably many critical points'', but from the viewpoint of submanifolds, such result belongs to the case of ``countably many submanifolds'' in our Theorem \ref{thm:escape3}. (A submanifold can contain uncountably many points, but is still a single object to escape.) This can also be inferred from the use of a countable subcover in Theorem 10 of \cite{panageas2016gradient} and the subsequent proof of the main theorem, where the convergence set to any saddle is categorized into a countable number of stable manifolds.

To further illustrate this point, here we give some interesting examples. The saddle sets in Example \ref{ex:counter} occupy a zero measure set in the whole manifold. They cannot be assembled into countably many connected submanifolds. We will analyze and see why they cannot be almost surely avoided. 

\begin{example}
\label{ex:counter}
    Let $\M = [-1, 2]\times [-1, 1] \subset \R^2$ be a rectangular region, viewed as a 2-dimensional submanifold of $\R^2$. Then the tangent space $T\M$ equals $\M$. To construct the function $f$ on $\M$, we need the 1-dimensional Smith-Volterra-Cantor (``fat Cantor'') set $V$ in $[0,1]$. The construction is as follows:
    \begin{enumerate}[label={\arabic*)}]
        \item Remove the middle interval of length $\frac{1}{4}$ from $[0,1]$, and the remaining set is $[0,\frac{3}{8}]\cup [\frac{5}{8},1]$;
        \item Remove 2 middle subintervals of length $\frac{1}{4^2}$ from the 2 remaining intervals, and the remaining set is $[0, \frac{5}{32}]\cup[\frac{7}{32}, \frac{3}{8}]\cup[\frac{5}{8}, \frac{25}{32}]\cup[\frac{27}{32},1]$;
        \item Remove 4 middle subintervals of length $\frac{1}{4^3}$ from the 4 remaining intervals;
        \item \ldots
    \end{enumerate}
    A visualization of the construction is given in Figure \ref{fig:SVC}. The construction is different from that of the classical Cantor set in that we remove proportionally shorter subintervals, instead of subintervals proportional to the mother interval. Therefore, $V$ has positive measure in $\R$, $meas(V) =\frac{1}{2}$, while the classical Cantor set has zero measure. Still, $V$ is nowhere dense. 
    
    We look for a synthetic objective function on $\M$ in the form
    \begin{align*}
        f: \M\to\R, \quad f(x,y) = -p(x)+y^2,
    \end{align*}
    where $p(x)$ is a function of certain regularity on the 1D interval $x\in[-1,2]$. Consider two examples:
    
    \begin{enumerate} [label={(\Alph*)}]
        \item Define $p_A(x) = 0$ for $x\in V$. As $V$ is a closed set, write $ V^c = [-1, 2]\backslash V = \left(\bigcup_\alpha (a_\alpha, b_\alpha)\right) \cup [-1,0) \cup (1,2]$ as the disjoint union of intervals. On each interval $(a,b)$, let
        \begin{align*}
        &p_A(x) = 
        \begin{cases}
            (x-a)^2, & \text{for} \quad a<x \le a+\frac{(b-a)}{4};\\
            C_1(x-\frac{a+b}{2})^4 + C_2(x-\frac{a+b}{2})^2 + C_3, & \text{for} \quad a+\frac{(b-a)}{4} < x \le b-\frac{(b-a)}{4};\\ 
            (b-x)^2, & \text{for} \quad b-\frac{(b-a)}{4} < x<b,
        \end{cases} 
        \end{align*}
        where $C_1 = \frac{8}{(b-a)^2}$, $ C_2 = -2$, $C_3 = \frac{5(b-a)^2}{32}$. See a visualization in Figures \ref{fig:interval} and \ref{fig:counter}.
        \item Similar to (A), but on each interval $(a,b)$, let 
        \begin{align*}
        &p_B(x) = 
        \begin{cases}
            (x-a)^4, & \text{for} \quad a<x \le a+\frac{(b-a)}{4};\\
            C_1(x-\frac{a+b}{2})^6 + C_2(x-\frac{a+b}{2})^4 + C_3, & \text{for} \quad a+\frac{(b-a)}{4} < x \le b-\frac{(b-a)}{4};\\ 
            (b-x)^4, & \text{for} \quad b-\frac{(b-a)}{4} < x<b,
        \end{cases} 
        \end{align*}
        where $C_1 =\frac{512}{3(b-a)^4} $, $ C_2 =-\frac{24}{(b-a)^2} $, $C_3 = \frac{11(b-a)^2}{96}$.
    \end{enumerate}
    
    It is easy to see that both functions $p_i(x)$, $i=A$ or $B$, satisfy $p(x)\ge 0$ and $p(x) = 0$ if and only if $x\in V$. Thus for $f_i(x,y) = -p_i(x)+y^2$, the saddle set of $f$ is $S = V\times [0]$. Viewed in the two-dimensional manifold, it has zero measure. But the converging set of $S$ is $W^s_S(\varphi) = V\times [-1,1]$. It has positive measure in $\M$: $meas(W^s_S(\varphi)) = 1$. If we start the PGD algorithm with a uniform random initialization, the probability that $\{\bZ_n\}_{n=0}^\infty$ end up towards a saddle is
    \begin{align*}
        Pr(\lim_{n\rightarrow\infty} \bZ_n\in S)=\frac{1}{6} > 0.
    \end{align*}
    
    So what happens? The reason that gradient descent fails to escape such a saddle set is well hidden. %An earlier version of this paper has taken Example (A) as a concrete counterexample that uncountable saddle set cannot be escaped, but the underlying reason turns out to be something else.  
    Specifically, in Example (A), $p_A(x)$ is only $C^1$ but not $C^2$. For each $x\in V$, the second derivatives of $p_A(x)$ on two sides are not equal. One side of $x$ is an open interval in $V^c$, so the second derivative is 2; while on the other side $x$ is the limit point of a sequence $\{x_j\}_{j=1}^\infty \subset V$, and the second derivative is not well-defined. As for Example (B), $p_B(x)$ is $C^3$ over $[-1,2]$ and thus $f_B(x,y)$ satisfies the regularity requirements. However, $p_B''(x) = 0$ $\forall x\in V$, so $x\in V$ are not \emph{strict} saddles.  
    
    A loosely relevant discussion of the above constructions can be found in \cite{rudin1964principles}, Exercise 5.21. This example is purely synthetic, but it raises a healthy warning as to how much the assumptions can be relaxed while the escape from saddle sets is still valid. 
\end{example}    
    
\begin{figure}[ht]
    \centering
    \begin{minipage}[c]{.35\linewidth}
        \begin{subfigure}[t]{\linewidth}
        \includegraphics[width = \linewidth, align = c]{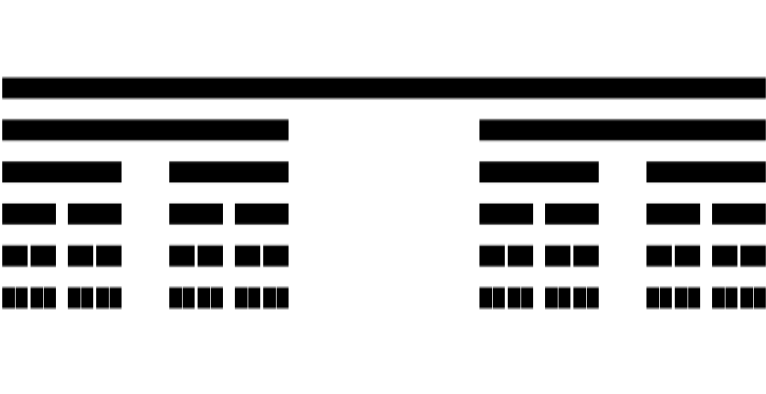}
        \caption{Construction of Smith-Volterra-Cantor set for the first 5 steps}
        \label{fig:SVC}
    \end{subfigure}
    \begin{subfigure}[t]{\linewidth}
        \includegraphics[width = \linewidth, align = c]{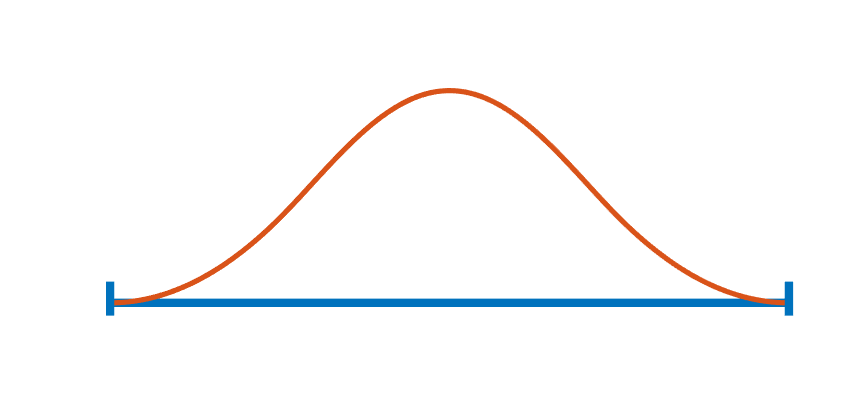}
        \caption{Example of $p_A(x)$ on a single interval}
        \label{fig:interval}
    \end{subfigure}
    \end{minipage}
    \hspace{0.08\textwidth}
    \begin{subfigure}[c]{.5\linewidth}
        \includegraphics[width = \linewidth, align = c]{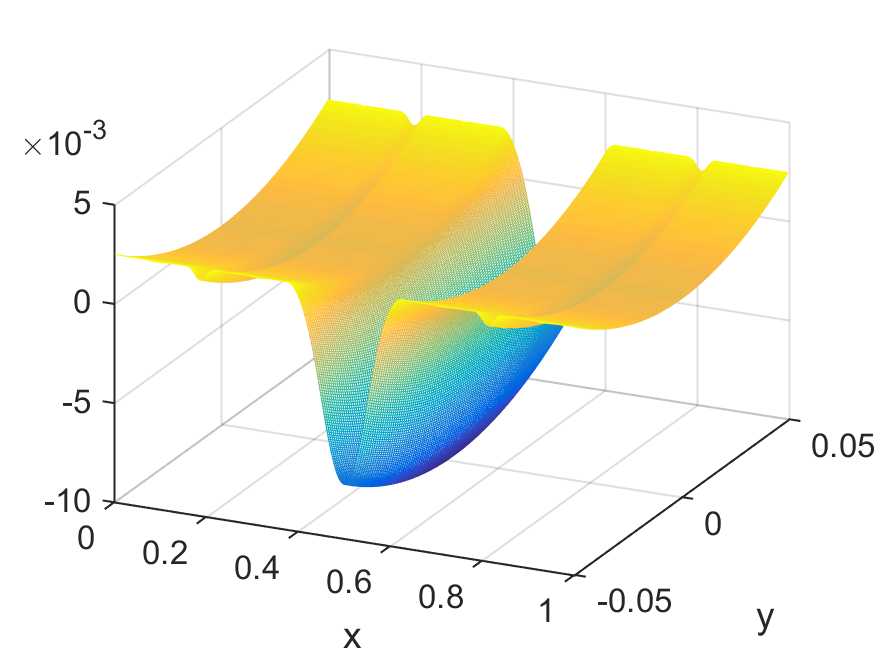}
        \caption{Visualization of $f(x,y) = -p_A(x) + y^2$, where saddle set is the middle of the ``plateaus''}
        \label{fig:counter}
    \end{subfigure}
    \caption{Illustration of Example \ref{ex:counter}}
\end{figure}

}

\section{Exploring the low-rank matrix manifold structure}
\label{sec:Mr}

The low-rank matrix manifold $\M_r:=\{\bX:rank(\bX) = r\}$ has long been studied in low-rank recovery problems, and there has been a glossary of previous works on various applications. %, for example \cite{ManiOpt1} and \cite{ManiOpt2}. 
Some real-world applications carry an intrinsic low-rank matrix structure with them, e.g. matrix sensing and matrix completion, while others were originally formulated as the optimization in the Euclidean space, e.g. the phase retrieval problem to be discussed in Section \ref{sec:pr}. 
One obvious advantage of manifold optimization in this case is that, instead of finding $\bU \in \R^{n\times r}$ such that $\bX=\bU\bU^\top$, solving $\bX$ directly on the manifold $\M_r$ itself helps avoid the duplication of spurious critical points (i.e. $\bU\bU^\top$ versus $(\bU \bm{R})(\bU \bm{R})^\top$, $\bm{R}$ unitary). Another more important advantage is that replacing $\bU\bU^\top$ with $\bX$ changes the form of the objective function, and it will drastically improve the convergence rate of a first-order method. The latter will be the topic of our upcoming paper \cite{paper2}.

This section is a self-contained discussion of some essential properties of the low-rank matrix manifold $\M_r$ and its closure $\overline{\M_r}$.  

\subsection{Manifold setting}
First, it is necessary to point out that most previous works are based on the fixed-rank manifold $\M_r$, but the following reasons show that the fixed-rank manifold is not rigorous enough:
\begin{enumerate}[label={(\arabic*)}]
    \item The fixed-rank manifold $\M_r$ is {\it not} closed. Optimization techniques like the gradient descent generate a sequence of points towards the ground truth, and closedness is naturally necessary for asymptotic convergence analysis.
    \item %Low-rank manifold optimization algorithms often involves truncated SVD, which might reduce the rank to less than $r$. Thus, it's not assured that the generated points will still lie in $\M_r$ after singular value thresholding. 
    It is possible that at some step $(\bZ_n - \alpha P_{T_{\bZ_n}} (\nabla f(\bZ_n)))$ happens to have rank lower than $r$ and falls out of $\M_r$.
    \item {The fixed-rank manifold $\M_r$ also rules out the  possibility that the ground truth matrix has (intrinsic) rank $\tilde{r}<r$. This prohibits \emph{over-approximation}, i.e. when prior knowledge on $\tilde{r}$ is not available and people choose a larger guess attempting to capture more information.}
\end{enumerate}  

Therefore, we propose to consider $\overline{\M_r}=\{\bZ: rank(\bZ)\leq r\}$ instead. This closure is obtained by taking the union of all lower rank manifolds $\M_s \, (s<r)$ together with the rank-$r$ manifold. On the one hand, the essential definitions and properties of $\M_r$ only needs to be slightly modified to accommodate $\overline{\M}_r$, for example by introducing the ``tangent cone'' to be defined below. On the other hand, this poses no actual challenge to numerical computation, as in practice the randomly generated points would fall into $\overline{\M_r} \backslash \M_r$ with probability 0, unless over-approximation is involved.

The following lemmas validate that the closure $\overline{\M_r}$ is a nice domain to consider.
Note that the structure of the manifolds can be different on the real/complex field, or with/without symmetry, as they take different subsets of the Euclidean space $\R^{m\times n}$. So they are listed separately. %Specifically, we have the following lemmas:

%\subsubsection{Real asymmetric case: $\overline{\M_r} \subset \R^{m*n}$}
\begin{lemma}[Real, asymmetric case] \leavevmode \label{lemma:manifoldinR}
Let $\overline{\M_r}=\{\bZ\in\R^{m\times n}: rank(\bZ)\leq r \}$, and $\M_s=\{\bZ\in\R^{m\times n}: rank(\bZ)=s\}$. Then
    \begin{enumerate}[label={(\arabic*)}]
        \item $\M_r$ is dense in $\overline{\M_r}$;
        \item $\M_r$ is connected;
        \item The local dimension of $\M_s$ is $(m+n-s)s$;
        \item The boundary of $\M_r$ is $\overline{\M_r}\setminus\M_r=\cup_{0\leq s<r}\M_s$.
    \end{enumerate}
\end{lemma}
\begin{proof}\leavevmode
    \begin{enumerate}[label={(\arabic*)}]
        \item For $\bZ\in\overline{\M_r}\setminus \M_r$, it can be approached by a sequence of rank-$r$ matrices $\{\bZ_k\}\subset \M_r$ such that $\lim_{k\rightarrow\infty}\bZ_k=\bZ$.
        \item Consider
            \begin{align*}
                \widetilde{\Phi_r}: \mathbf{SO}(m,\R)\times\R_+^r\times \mathbf{SO}(n,\R)&\rightarrow \R^{m\times n}\\
                (\widetilde{\bm{U}}, \bm{\sigma}_r, \widetilde{\bm{V}}) &\mapsto \widetilde{\bm{U}} \widetilde{\bm{\Sigma}} \widetilde{\bm{V}}^\top
            \end{align*}
            where $\mathbf{SO}(m,\R)$ is the real orthogonal group in dimension $m$, $\bX=\widetilde{\bm{U}} \widetilde{\bm{\Sigma}} \widetilde{\bm{V}}^\top$ is the full singular value decomposition of $\bX$, $\bm{\sigma}_r \in \R^r$, $\bm{\sigma}_r(i)\ne 0$, $i=1,\cdots, r$, and
            \begin{align*}
                \widetilde{\bm{\Sigma}}=\begin{pmatrix}
                \text{diag}(\bm{\sigma}_r) & \\
                & \bm{0}_{(m-r)\times(n-r)}
                \end{pmatrix}.
            \end{align*}
        Since $\mathbf{SO}(m,\R)\times\R_+^r\times \mathbf{SO}(n,\R)$ is connected and $\widetilde{\Phi_r}$ is continuous, its orbit $\M_r$ is connected. 
        \item Consider the compact singular value decomposition $\bX = \bm{U}\text{diag}(\bm{\sigma}_s)\bm{V}^\top$, where $\bm{U}\in\R^{m*s}$, $\bm{V}\in\R^{n*s}$, and $\bm{\sigma}_s$ is properly ordered. %With the mapping 
        %\begin{align*}
            %\tilde{\Phi}_s: (\tilde{\bm{U}}, \bm{\sigma}_s, \tilde{\bm{V}}) \mapsto \tilde{\bm{U}}\text{diag}(\bm{\sigma}_s)\tilde{\bm{V}}^\top,
        %\end{align*} 
        %for the rank-$s$ manifold $\M_s$, 
        The local dimension is $s + \frac{(2m-s-1)s}{2} + \frac{(2n-s-1)s}{2}=(m+n-s)s$.
        \item It is obviously true.
    \end{enumerate}
\end{proof}

%\begin{remark}
%Noticing the dimension of $\M_r$ is $(m+n-r)r$, if one considers braches boundary $\overline{\M_r}\setminus\M_r$, eg: $\bZ\in\M_s$ with $s<r$, the submanifold $\M_s$ has lower dimension $(m+n-s)\times s<(m+n-r)\times r$ (Assume $m+n>>r+s$). Therefore, in the following we should modified the usual tangent space to be tangent cone to make our theory complete.
%\end{remark}

%The natural tangent space for any $\M_s$ and tangent cone for the whole $\overline{\M_r}$ as following, see also \cite{ManiOpt1}\cite{matrixLs}\cite{DyLowRank}\cite{Reinhold}.

%\subsubsection{Complex non-Hermitian Case: $\overline{\M_r} \subset \C^{m*n}$}
\begin{lemma}[Complex, non-Hermitian case]\leavevmode 
 Let $\overline{\M_r}=\{\bZ\in\C^{m\times n}: rank(\bZ)\leq r \}$, and $\M_s=\{\bZ\in\C^{m\times n}: rank(\bZ)=s\}$. Then
    \begin{enumerate}[label={(\arabic*)}]
        \item $\M_r$ is dense in $\overline{\M_r}$; 
        \item $\M_r$ is connected;
        \item The local dimension of $\M_s$ is $(2m+2n-s)s$;
        \item The boundary of $\M_r$ is $\overline{\M_r}\setminus\M_r=\cup_{0\leq s<r}\M_s$.
    \end{enumerate}
\end{lemma}
\begin{proof} \leavevmode
\begin{enumerate}[label={(\arabic*)}]
    \setcounter{enumi}{2}
    \item Consider the compact singular value decomposition $\bX = \bm{U}\text{diag}(\bm{\sigma}_r)\bm{V}^*$. The local dimension is $s+\frac{(4m-s-1)s}{2} + \frac{(4n-s-1)s}{2}=(2m+2n-s)s$.
\end{enumerate}
%    \begin{enumerate}
%        \item It is not hard to check that, for any  $\bZ\in\overline{\M_r}$, one can find a sequence $\{\bZ_k\}\subset \M_r$ such that $\lim_{k\rightarrow\infty}\bZ_k=\bZ$.
%        \item Consider
%            \begin{align*}
%                \Phi_r: \mathbf{SO}(m,\C)\times\R_+^r\times \mathbf{SO}(n,\C)&\rightarrow \C^{m\times n}\\
%                (\bm{U}, \bm{\sigma}_r, \bm{V}) &\mapsto \bm{U} \bm{\Sigma}_r \bm{V}^*
%            \end{align*}
%            where $\mathbf{SO}(m,\C)$ is the complex orthogonal group in dimension $m$, $\bm{U} \bm{\Sigma}_r \bm{V}^*$ is the singular value decomposition of $\bX=\bm{U} \bm{\Sigma}_r \bm{V}^*$, $\bm{\sigma}_r \in \R^r$, $\bm{\sigma}_r(i)\ne 0$, $i=1,\cdots, r$, and
%            \begin{align*}
%                \bm{\Sigma}_r=\begin{pmatrix}
%                diag(\bm{\sigma}_r) & \\
%                & \bm{0}_{(m-r)\times(n-r)}
%                \end{pmatrix}.
%            \end{align*}
%        Since $\mathbf{SO}(m,\C)\times\R_+^r\times \mathbf{SO}(n,\C)$ is connected and $\Phi_r$ is continuous, its orbit $\M_r$ is connected. 
%        \item Consider the compact singular value decomposition $\bX = \tilde{\bm{U}}\text{diag}(\bm{\sigma}_r)\tilde{\bm{V}}^*$.         The local dimension is  $s+\frac{(4m-s-1)s}{2}+\frac{(4n-s-1)s}{2}=(2m+2n-s)s$.
%        \item It is obviously true.
%    \end{enumerate}
\end{proof}

%The definitions of tangent space and tangent cone are similar to those in the real case.

%\subsubsection{Hermitian case: $\overline{\M_r} \subset \mathbb{S}^{n}$}
The Hermitian low-rank matrix manifold is also used very often as it fits the assumptions of many applications. But its structure is somewhat different from general non-Hermitian case because it has branches, see also \cite{matrixLs}. The real symmetric case is very similar and we omit the details.

\begin{lemma}[Complex, Hermitian case]\leavevmode 
 Let $\overline{\M_r}=\{\bZ\in\mathbb{S}^{n}(\C):  rank(\bZ)\leq r \}$, and $\M_s = \{\bZ\in\mathbb{S}^{n}(\C): rank(\bZ)= s \}$. Then
    \begin{enumerate}[label={(\arabic*)}]
        \item $\M_r$ is dense in $\overline{\M_r}$;
        \item $\M_r$ has $r+1$ disjoint branches and each branch is connected;
        \item The local dimension of $\M_s$ is $\frac{(4m-s+1)s}{2}$;
        \item The boundary of $\M_r$ is $\overline{\M_r}\setminus\M_r=\cup_{0\leq s<r}\M_s$.
    \end{enumerate}
\end{lemma}
\begin{proof} \leavevmode
    \begin{enumerate}[label={(\arabic*)}]
        \setcounter{enumi}{1}
        \item Consider the set of matrices that has $p$ positive eigenvalues and $q$ negative eigenvalues, $p+q=r$. Define
            \begin{align*}
                \Psi_{p,q}: \mathbf{GL^+}(n,\C) &\rightarrow \mathbb{S}^{n}\\
                \bm{P} &\mapsto \bm{P} \bm{I}_{p,q} \bm{P}^*
            \end{align*}
            where $\mathbf{GL^+}(n,\C)$ is the complex positive-determinant group in dimension $n$, and
            \begin{align*}
                \bm{I}_r=\begin{pmatrix}
                \bm{I}_p & & \\
                & -\bm{I}_q & \\
                & & \bm{0}_{(n-r)\times(n-r)}
                \end{pmatrix}.
            \end{align*}
        Thus, the orbit of each $\Psi_{p,q}$ is connected. The tuple $(p,q)$ is called the \emph{signature} of the matrix. However, matrices with different signatures are not path-connected on $\M_r$ (they are path-connected only on $\overline{\M_r}$). So $\M_r$ has $r+1$ branches corresponding to the orbits of $\Psi_{r,0}, \Psi_{r-1,1}, \cdots, \Psi_{0,r}$.
        \setcounter{enumi}{2}
        \item Consider $\bX\in\M_s$, let $\bX=\bm{U}\bm{D}_s\bm{U}^*$ be its compact eigenvalue decomposition with eigenvalues properly ordered. Consider the mapping
        \begin{align*}
            \Phi_s:    
            %\mathbf{SO}(n,\C)\times\R_+^s &\rightarrow \mathbb{S}^{n}\\
            (\bm{U}, \bm{\sigma}_s) \mapsto \bm{U} \bm{D}_s \bm{U}^*.
        \end{align*}
        %Then the above mapping $\tilde{\Phi}_s$ is a bijection. So
        The local dimension is  $s+\frac{(4m-s-1)s}{2}=\frac{(4m-s+1)s}{2}$.
    \end{enumerate}
\end{proof}

%*$\M_r$ has $r+1$ branches, local dimensions of $\overline{\M_r}$

\subsection{Geometric properties of the manifold}

Recall that the projected gradient descent (PGD) is defined as
\begin{equation}
    \bZ_{n+1}=\mathcal{R}\left(\bZ_n - \alpha_n P_{T_{Z_n}}(\grad f(\bZ_n))\right),
\end{equation}
which involves the embedded gradient $\grad f$, the tangent space projection $P_{T_{Z_n}}$, and the retraction $\mathcal{R}$ back onto the manifold. %When applied to the low-rank matrix manifold, we need to understand some geometric properties. 

For the low-rank matrix manifold, let it inherit the metric from its ambient Euclidean space, i.e. $\langle A, B \rangle = \text{trace}(A^\top B)$ and $\|A\| = \|A\|_F$. The tangent space is defined as follows.

\begin{definition}[Tangent space, non-Hermitian case]
Let $\bX\in\M_r$, $\bX=U\Sigma V^\top$ (or $\bX=U\Sigma V^*$). Let $\mathcal{U}=\text{Col}(U)$, $\mathcal{V}=\text{Col}(V)$ be the column spaces of $U$ and $V$ respectively. Then the \emph{tangent space} of $\M_r$ at $\bX$ is 
    \begin{equation*}
        T_X\M_r=(\mathcal{U}\otimes\mathcal{V})\oplus(\mathcal{U}\otimes\mathcal{V}^\perp)\oplus (\mathcal{U}^\perp\otimes\mathcal{V}).
    \end{equation*}
The projection operator onto the tangent space is:
    \begin{equation*}
        P_{T_{X}}=P_{\mathcal{U}}\otimes I + I\otimes P_{\mathcal{V}}- P_{\mathcal{U}}\otimes P_{\mathcal{V}}.
    \end{equation*}
\end{definition}

\begin{definition}[Tangent space, Hermitian case]
Let $\bX\in\M_r$, $\bX=U D U^\top$ (or $\bX=U D U^*$), $\mathcal{U}=\text{Col}(U)$. Then the \emph{tangent space} of $\M_r$ at $\bX$ is 
    \begin{equation*}
        T_X\M_s=(\mathcal{U}\otimes\mathcal{U})\oplus(\mathcal{U}\otimes\mathcal{U}^\perp)\oplus (\mathcal{U}^\perp\otimes\mathcal{U}).
    \end{equation*}
The projection onto the tangent space is
    \begin{equation*}
        P_{T_{\bX}\M_r}=P_{\mathcal{U}}\otimes I + I\otimes P_{\mathcal{U}}- P_{\mathcal{U}}\otimes P_{\mathcal{U}}.
    \end{equation*}
\end{definition}

%\begin{remark}
%    For $\bX\in\M_s\subset(\overline{\M_r}\setminus\M_r)$, the tangent space of $\overline{\M_r}$ at point $\bX$ cannot be inherited directly from the ``natural'' tangent space of $\M_s$ because it is dimensionally deficient. That is why we need the notion of ``tangent cone'' to glue all $\M_s$ ($s\le r$) together.
%\end{remark}

Since $\overline{\M_r}$ is constructed by ``gluing together'' all lower-rank matrix manifolds, it needs some special treatment at $\M_s$ ($s<r$) in order to make up for the deficient dimension. In addition to the classical tangent space \cite{ManiOpt1}\cite{DyLowRank}, we need the \emph{tangent cone} at these lower-dimensional instances \cite{Reinhold}.

\begin{definition}[Tangent cone, non-Hermitian case]
    Let $\bX\in\M_s\subset(\overline{\M_r})$ where $s<r$,  $\bX=U\Sigma V^\top$ (or $\bX=U\Sigma V^*$), $\mathcal{U}=\text{Col}(U)$, $\mathcal{V}=\text{Col}(V)$. Then the \emph{tangent cone} of $\overline{M_r}$ at $\bX$ is 
    \begin{equation*}
        T_X\overline{\M_r}=T_X\M_s\oplus\{\eta: \eta\in \mathcal{U}^\perp\otimes\mathcal{V}^\perp, \text{ rank}(\eta)=r-s\}.
    \end{equation*}
    The projection onto the tangent cone is the projection onto the tangent space plus a rank (r-s) principal component, i.e.
    \begin{align*}
        P_{T_{\bX}\overline{\M_r}}(Y) = P_{T_{\bX}\M_s}(Y) + Y_{r-s}
    \end{align*}
    where $Y_{r-s}$ is a best rank (r-s) approximation of $Y-P_{T_{\bX}\M_s}(Y)$ in the Frobenious norm.
\end{definition}

\begin{definition}[Tangent cone, Hermitian case]
    Let $\bX\in\M_s\subset(\overline{\M_r})$ where $s<r$, $\bX=U D U^\top$ (or $\bX=U D U^*$), $\mathcal{U}=\text{Col}(U)$. Then the \emph{tangent cone} of $\overline{M_r}$ at $\bX$ is 
    \begin{equation*}
        T_X\overline{M_r}=T_X\M_s\oplus\{\eta: \eta\in \mathcal{U}^\perp\otimes\mathcal{U}^\perp, \text{ rank}(\eta)=r-s\}.
    \end{equation*}
    The projection onto the tangent cone is 
    \begin{align*}
        P_{T_{\bX}\overline{\M_r}}(Y) = P_{T_{\bX}\M_s}(Y) + Y_{r-s}
    \end{align*}
    where $Y_{r-s}$ is a best rank (r-s) approximation of $Y-P_{T_{\bX}\M_s}(Y)$ in the Frobenious norm.
\end{definition}

%For simplicity of notations we do not distinguish between the tangent space and the tangent cone in the following. We also use $P_{T_{\bX}}$ for both $P_{T_{\bX}\M_r}$ and $P_{T_{\bX}\overline{\M_r}}$ when there is no confusion.

The retraction, or the projection onto the manifold as it appears in some literature, is defined as follows.

\begin{definition}[Retraction]
    The \emph{natural retraction} on $\M_r$ is defined as
    \begin{align*}
        \mathcal{R}_N(\bZ+\xi) = \mathop{argmin}_{\bm{Y} \in \M_r} \|\bZ+\xi-\bm{Y}\|_F .
    \end{align*}
\end{definition}
In other words, $\mathcal{R}_N(\bZ+\xi)$ is the best rank-$r$ approximation of $\bZ+\xi$. Such a retraction not only satisfies the first-order retraction property (\ref{eq:first_order_retraction}), but is actually second-order, i.e.
    \begin{align*}
        \mathcal{R}(\bZ+\alpha\xi) = \bZ+\alpha\xi + \mathcal{O}(\alpha^2).
    \end{align*}
Vandereycken in \cite{ManiOpt1} provides an explicit second-order approximation $\mathcal{R}_N^{(2)}$ to this second-order retraction $\mathcal{R}_N(\cdot)$ and shows that $\mathcal{R}_N(\bZ+\xi) = \mathcal{R}_N^{(2)}(\bZ+\xi) + \mathcal{O}(\|\xi\|^3)$.

It can also been seen from the above definition that projected version of the gradient descent is also cheaper in computation. Namely, to solve $\mathcal{R}_N(\bZ_n - \alpha_n \grad f(\bZ_n))$ involves solving the SVD of a rank-$n$ matrix, while $\mathcal{R}_N(\bZ_n - \alpha_n P_{T_{\bZ_n}}(\grad f(\bZ_n)))$ only involves that of a rank $2r$ matrix. The latter is more favorable when $r\ll n$.

%\vspace{6pt}

%The above discussion of the low-rank matrix manifold lays the foundation of our analysis for a series of problems including matrix sensing and phase retrieval in \cite{paper2}. Put together under a unified framework of low-rank matrix optimization, these optimization problems demonstrate surprisingly good convergence rate with the PGD. Combined with asymptotic convergence to minimizers and escape of strict saddles as in Section \ref{sec:PGD} in this paper, these results give a complete picture of the landscape of these optimization problems on the manifold. 

The above discussion of the low-rank matrix manifold lays down the foundation of our analysis both in Section \ref{sec:pr} about phase retrieval and in the upcoming paper \cite{paper2} on more general problems. While Section \ref{sec:pr} mainly focuses on $\M_1$ and studies the asymptotic convergence, the upcoming paper will make a heavier use of $\overline{\M_r}$ and look at the convergence rate side of the problem. We will put multiple problems under the same framework and give a complete picture of the landscape of these optimization problems.

\section{Asymptotic escape on the low-rank matrix manifold}
\label{sec:pr}

%1. https://static.googleusercontent.com/media/research.google.com/zh-CN//pubs/archive/36898.pdf

%In this section, we demonstrate the asymptotic convergence of PGD on Riemannian manifold and escape of strict saddles through some concrete applications. 

In this section, we consider the {phase retrieval problem} \rev{\cite{cai2018solving}\cite{PhaseRetrieval2}\cite{PhaseRetrieval}} on the rank-1 matrix manifold. This serves both as an application of our asymptotic escape analysis for strict saddles, and as a demonstration of the possibility of treating such problem rigorously on the manifold as opposed to the Euclidean space.

Since the phase retrieval problem involves a large number of stochastic measurements (i.e. random coefficient matrices $\{A_j(\omega)\}$ that constitute the objective function $f_\omega$, $\omega$ indicating the random event), we will approach this problem in two steps. First, a crude analysis will be performed on its expectation $\mathbb{E}_\omega f$. In this case we will locate a strict critical submanifold in the shape of a ``hyper ring''. Then, for the non-expectation case $f = f_{(\omega)}$, we will prove a rather surprising result that it almost surely has only a finite number of saddle points. We will then show that with high probability, these saddles are strict saddles, and we know they are located near the above ``hyper ring'', so our asymptotic escape analysis is also applicable. The asymptotic escape is further supported by numerical experiments.

\subsection{Phase retrieval on manifold: the expectation}
The problem of phase retrieval \rev{in the case of real values} aims to retrieve the information about $x\in\R^n$, from the phaseless measurements
\begin{align*}
    y_j = |a_j^\top x|^2, \quad j=1,\ldots m,
\end{align*}
where the entries of $\{a_j(\omega)\}_{j=1}^m$ are drawn from i.i.d. Gaussian. Usually a large $m$ is needed to ensure successful recovery of $x$.

Let $\bX = xx^\top$, $A_j = a_ja_j^\top$, then $y_j = \langle A_j, \bX \rangle$. The problem can be posed on the rank-1 matrix manifold $\M_1$ as
\begin{align*}
    \min_{\bZ\in\M_1} f(\bZ) := \frac{1}{2m}\sum_{j=1}^m |\langle A_j, \bZ-\bX \rangle|^2.
\end{align*}

We can apply the projected gradient descent to solve this problem on $\M_1$. \rev{We refer the reader to \cite{cai2018solving} in which the authors discussed the practical aspects of the PGD algorithm applied to phase retrieval.} It is easily seen that $\bZ=\bX$ is the unique global minimizer. To ensure asymptotic convergence of the PGD to the global minimizer, it remains to rule out local minimizers and identify other critical points as strict saddles. Previous works \cite{PhaseRetrieval2}\cite{PhaseRetrieval} have shown that phase retrieval has no spurious local minimum at least with high probability \rev{in the Euclidean setting}. The analysis of saddle has been more complicated because of the stochasticity and Euclidean space parameterization.

It helps to take the expectation of $f(\bZ)$ and look into its landscape on the manifold. Note that 
\begin{align*}
    \bar{f}(\bZ):=\mathbb{E}_\omega f(\bZ) = \frac{3}{2}\|\bZ\|_F^2 + \frac{3}{2}\|\bX\|_F^2 - \|\bZ\|_F\|\bX\|_F - 2\langle\bZ, \bX \rangle,
\end{align*}
and the Riemannian gradient (i.e. projected gradient) is
\begin{align*}
    \text{grad}\bar{f}(\bZ) = P_{T_{\bZ}} (\nabla \bar{f}(\bZ)) = P_{T_{\bZ}}((3-\frac{\|\bX\|_F}{\|\bZ\|_F})\bZ - 2\bX).
\end{align*}
The first-order condition is satisfied if either $\bZ = \bX$, or
\begin{align*}
    \|\bZ\|_F = \frac{1}{3}\|\bX\|_F, \quad \langle\bZ, \bX\rangle = 0.
\end{align*}
The latter are spurious critical points, and they form a ($n$-2)-dimensional submanifold on $\M_1$. To see whether they are strict saddles, we look into their Hessian.

Let $\bZ = zz^\top$, $u = z/\|z\|_2$, then $u\perp x$. Any element $\xi\in T_{\bZ}$ can be represented as $\xi = w uu^\top + uv^\top + vu^\top$, where $w\in\R$, $v\in\R^n$ and $v\perp u$. From \cite{ManiOpt1}, $R_N(\bZ+\xi) = \bZ + \xi +\eta + \mathcal{O}(\|\xi\|^3)$ where $\eta = vv^\top/\|\bZ\|_F$. Using the formula that $\text{Hess} f(\bZ) = \text{Hess} (f\circ \mathcal{R}_{\bZ})(t\xi)\mid_{t = 0}$, we have
\begin{align*}
    f\circ \mathcal{R}_{\bZ}(\xi) &= f(\bZ + \xi +\eta) + \mathcal{O}(\|\xi\|^3) \\
    &= f(\bZ) + \langle \nabla f(\bZ), \xi \rangle + \langle \nabla f(\bZ), \eta \rangle + \frac{1}{2}\langle \nabla^2 f(\bZ)[\xi], \xi \rangle + \mathcal{O}(\|\xi\|^3), 
\end{align*}
and collecting second order term gives
\begin{align*}
    &\langle \text{Hess}\bar{f}(\bZ)[\xi], \xi \rangle = 2\langle \nabla \bar{f}(\bZ), \eta \rangle + \langle \nabla^2 \bar{f}(\bZ)[\xi], \xi \rangle \\
    &= (6-2\frac{\|\bX\|_F}{\|\bZ\|_F})\langle \bZ, \eta \rangle - 4\langle \bX, \eta \rangle + (3-\frac{\|\bX\|_F}{\|\bZ\|_F})\|\xi\|_F^2 + \frac{\|\bX\|_F}{\|\bZ\|_F^3}\langle \bZ, \xi \rangle^2 \\
    &= - 4\langle \bX, \eta \rangle + \frac{3}{\|\bZ\|_F^2}\langle \bZ, \xi \rangle^2 \\
    &= -4 \frac{|x^\top v|^2}{\|\bZ\|_F} + 3w^2.
\end{align*}
Let $\xi = ux^\top + xu^\top$, then $\langle \text{Hess}\bar{f}(\bZ)[\xi], \xi \rangle = -12\|\bX\|_F < 0$. Therefore these spurious critical points are strict saddles. In fact they form a strict critical submanifold $\mathcal{N}=\{\bZ\in \M \mid \|\bZ\|_F = \frac{1}{3}\|\bX\|_F, \quad \langle\bZ, \bX\rangle = 0\}$. For $p\in\N$, $T_p\M = dim(T_p\N) = n-2$, $dim(\nu_p^s\M) = dim(\nu_p^u\M) = 1$. PGD will escape the strict critical submanifold and converge to the minimum of $\bar{f}$ almost surely by Theorem \ref{thm:escape3}.

Note that although we focus on the real case (i.e. $\M_1(\R)$) here, the above results can be generalized to the complex case easily, and the only change is in the constants concerning Gaussian moments.

\subsection{Phase retrieval: dive into specific realizations}

Specific realizations of phase retrieval may have much more complicated landscape than the expectation case. However, in the previous work \cite{PhaseRetrieval2} the authors have shown that \rev{for a slightly modified objective function}, with high probability, the saddles of a specific realization of phase retrieval lie in the neighborhood of the above $\N$, the so-called ``hyper ring''.

Consider 
\begin{align*}
     f(\bZ) = \frac{1}{2m}\sum_{j=1}^m |\langle A_j, \bZ-\bX \rangle|^2
\end{align*}
for a specific realization of $\{A_j(\omega)\}_{j=1}^m$. The Riemannian gradient is
\begin{align*}
    \text{grad}f(\bZ) = P_{T_{\bZ}} (\nabla f(\bZ)) = \frac{1}{m}\sum_{j=1}^m \langle A_j, \bZ-\bX \rangle P_{T_{\bZ}}(A_j).
\end{align*}
And the Riemannian Hessian is 
\begin{align*}
    \langle \text{Hess}f(\bZ)[\xi], \xi \rangle &= 2\langle \nabla f(\bZ), \eta \rangle + \langle \nabla^2 f(\bZ)[\xi], \xi \rangle\\
    &=\frac{1}{m}\sum_{j=1}^m (2\langle A_j, \bZ-\bX \rangle \langle A_j, \eta \rangle + \langle A_j, \xi \rangle^2).
\end{align*} 

The first result is a rather surprising one showing the finite number of critical points for phase retrieval.

\begin{lemma}
\label{lemma:finite}
    When $m \ge n$, the above $f(\bZ)$ almost surely has only finite number of critical points on the manifold $\M_1$.
\end{lemma}

The proof of Lemma \ref{lemma:finite} is quite neat using the following result from \cite{garcia1980number} and restated in \cite{li1987solving}.

\begin{lemma}
\label{lemma:TYLi}
    For a polynomial system $P(x) = (p_1(x), \ldots, p_n(x))$ with $x = (x_1, \ldots x_n)$ and $d_i = \text{degree } p_i(x)$, let $p_i(x) = p_i^1(x)+p_i^2(x)$ where $p_i^1(x)$ consists of all the terms of $p_i(x)$ with degree $d_i$. If the homogeneous polynomial system $P^1(x) = (p_1^1(x), \ldots, p_n^1(x)) = 0$ has only the trivial solution $x=0$, then the original system $P(x)=0$ only has a finite number of solutions. Moreover, the number of solutions is exactly $\Pi_{i=1}^n d_i$.
\end{lemma}

\begin{proof}[Proof of Lemma \ref{lemma:finite}]
    The first-order condition $\text{grad}f(\bZ) = 0$ is equivalent to
    \begin{align*}
        \frac{1}{m}\sum_{j=1}^m \langle A_j, \bZ-\bX \rangle P_{T_{\bZ}}(A_j)=0.
    \end{align*}
    Let $\tilde{U} \in \R^{n\times(n-1)}$ be the orthonormal complement of $u$. Then we have
    \begin{align*}
        P_{T_{\bZ}}(A_j) &= uu^\top A_j uu^\top + uu^\top A_j \tilde{U} \tilde{U}^\top + \tilde{U} \tilde{U}^\top A_j uu^\top \\
        &= u (a_j^\top u)^2 u^\top + u (a_j^\top u \cdot a_j^\top \tilde{U}) \tilde{U}^\top + \tilde{U} (a_j^\top u \cdot \tilde{U}^\top a_j) u^\top.
    \end{align*}
    Applying a basis transform $(u,\tilde{U})$ to the first-order condition, by symmetry, it is equivalent to
    \begin{align*}
        \begin{cases}
            \frac{1}{m}\sum_{j=1}^m \langle A_j, \bZ-\bX \rangle \cdot a_j^\top u \cdot a_j^\top u =0, \\
            \frac{1}{m}\sum_{j=1}^m \langle A_j, \bZ-\bX \rangle \cdot a_j^\top u \cdot a_j^\top \tilde{U} = 0,
        \end{cases}
    \end{align*}
    which is equivalent to
    $
        \sum_{j=1}^m \langle A_j, \bZ-\bX \rangle (a_j^\top u) a_j =0,
    $
    i.e. finding $z\in\R^n$ such that
    \begin{align}
    \label{eq:pr_first_order}
        \sum_{j=1}^m (|a_j^\top z|^2 - |a_j^\top x|^2) (a_j^\top z) a_j =0.
    \end{align}
    This is a third-order heterogeneous polynomial system of $n$ equations for $n$ unknowns. The homogeneous part of the system is \begin{align*}
        \sum_{j=1}^m |a_j^\top z|^2 (a_j^\top z) a_j =0.
    \end{align*}
    This system almost surely only has the trivial solution $z=0$. To see this, note that it requires $\sum_{j=1}^m |a_j^\top z|^4=0$, i.e.
    \begin{align*}
        a_j^\top z = 0, \quad j = 1, \ldots, m.
    \end{align*}
    Since $\{a_j\}$ are i.i.d. Gaussian, when $m\ge n$ this linear system is almost surely nondegenerate. 
    Now we can apply Lemma \ref{lemma:TYLi} and deduce that the original system only has finite number of solutions, i.e. $f(\bZ)$ only has finite number of critical points on the manifold.
\end{proof}

\begin{remark}
    From equation (\ref{eq:pr_first_order}), we can see that the first-order condition on the manifold $\M_1$ is equivalent to that in the parameterized Euclidean space. This means that their critical points match. Still, a critical point $\bZ = zz^\top$ corresponds to at least two critical points $\pm z$ in the parameterized Euclidean space. Also, their Hessian can be very different. 
\end{remark}
\rev{
\begin{remark}
    The result of Lemma \ref{lemma:finite} only applies to the case $z\in\R^n$. In the case $z \in \C^n$, we conjecture that there would be a finite number of of critical submanifolds instead. Each critical submanifold consists of $\{e^{i\theta} z_*: \theta \in [0,2\pi)\}$, the family of \emph{phaseless} vectors. %Although the result of Lemma \ref{lemma:finite} is not directly applicable to the complex case (as Equation (\ref{eq:pr_first_order}) is no longer a polynomial system), it only needs a slight modification. 
    To see this, we can impose the constraints $a_j^H z \in \R$ to the above equations (this is always possible by letting $a_j$ absorb the phase information, which does not alter $A_j$). Now we can replace $|\cdot|$ with $(\cdot)$ and again get a polynomial system. Lemma \ref{lemma:TYLi} is still applicable, and we get the finiteness of solutions on this constrained subset. To remove the constraints, we put the phase information back and obtain the submanifolds.
\end{remark}
}

The Hessians of saddle points in phase retrieval are treated in the next lemma. \rev{Note that the condition $m\ge n$ in Lemma \ref{lemma:finite} only ensures the finite number of saddles. To make sure that saddles are strict, we need $m \gtrsim n \log n$, which is consistent with recovery guarantees from previous works (see e.g. \cite{cai2018solving} and references therein)}.

\begin{theorem}
\label{thm:prnegative}
    Given $\delta_0, \delta_1 >0$. If $m\geq C(\delta_1)n\log n$, then with high probability no less than $1-\frac{C_1}{m}-e^{-C_2n}$, for all $\bZ$ that satisfy the following conditions
    $$\begin{cases}
     &\langle \bZ,\bX\rangle \le \delta_0\|\bZ\|_F\|\bX\|_F, \\
     &\frac{1}{3}-\delta_0 \le\frac{\|\bZ\|_F}{\|\bX\|_F}\le\frac{1}{3}+\delta_0, \\
     & P_{T_{\bZ}}(\grad f(\bZ))=0,
     \end{cases}$$\\
     we have
     $$
     \lambda_{min}(\text{Hess } f(\bZ)) \le \Lambda(\delta_0, \delta_1) < 0.
     $$
    Here $C_1$, $C_2$ are absolute constants, $C(\delta_1)$ depend only on $\delta_1$, and $\Lambda$ depend only on $\delta_0$ and $\delta_1$. If we further require $\delta_0 < \frac{1}{6}$, $\delta_1 < \frac{5}{36}$, then $\lambda_{min}(\text{Hess} f(\bZ)) < -1$.
\end{theorem}

\begin{proof}[Proof of Theorem \ref{thm:prnegative}]
The construction of a negative curvature direction is similar to that in the previous subsection. Let $\xi = xu^\top+ux^\top$, then $\xi \in T_{\bZ}$. Since now $x$ and $z$ are not orthogonal, $\xi = wuu^\top + uv^\top+vu^\top$, where $w = 2u^\top x$ and $v= x - uu^\top x$. The Hessian is
\begin{align*}
    \langle \text{Hess}f(\bZ)[\xi], \xi \rangle &=\frac{1}{m}\sum_{j=1}^m (2\langle A_j, \bZ-\bX \rangle\langle A_j, \eta \rangle + \langle A_j, \xi \rangle^2) \\
    &=\frac{1}{m}\sum_{j=1}^m (2\langle A_j, \bZ-\bX \rangle\langle A_j, \frac{xx^\top}{\|\bZ\|_F} + (\eta-\frac{xx^\top}{\|\bZ\|_F}) \rangle + \langle A_j, \xi \rangle^2)
\end{align*} 
An important observation is 
\begin{align*}
    \frac{1}{m}\sum_{j=1}^m (2\langle A_j, \bZ-\bX \rangle\langle A_j, (\eta-\frac{xx^\top}{\|\bZ\|_F}) \rangle = 0.
\end{align*}
This is because 
\begin{align*}
    \eta\cdot\|\bZ\|_F-xx^\top &= vv^\top - xx^\top = (x - uu^\top x)(x - uu^\top x)^\top - xx^\top \\
    &= - uu^\top xx^\top - xx^\top uu^\top + uu^\top xx^\top uu^\top \in T_{\bZ},
\end{align*}
and the first-order condition gives $\frac{1}{m}\sum_{j=1}^m \langle A_j, \bZ-\bX \rangle\langle A_j, \zeta \rangle = 0$ for any $\zeta \in T_{\bZ}$.

Therefore, we have
\begin{align*}
    \frac{\langle \text{Hess}f(\bZ)[\xi], \xi \rangle}{\|\xi\|_F^2} 
    &= \frac{\frac{1}{m}\sum_{j=1}^m (2\langle A_j, \bZ-\bX \rangle\langle A_j, \frac{xx^\top}{\|\bZ\|_F} \rangle + \langle A_j, \xi \rangle^2)}{\|\xi\|_F^2} \\
    &= \frac{\frac{1}{m}\sum_{j=1}^m (2(|a_j^\top z|^2-|a_j^\top x|^2)|a_j^\top x|^2 + 4|a_j^\top z|^2 |a_j^\top x|^2)}{2(\|z\|^2\|x\|^2+\langle x,z\rangle^2)} \\
    &=\frac{\frac{1}{m}\sum_{j=1}^m (3|a_j^\top z|^2|a_j^\top x|^2- |a_j^\top x|^4)} {\|z\|^2\|x\|^2+\langle x,z\rangle^2}.
\end{align*}
Using the concentration inequalities from Section 4 in \cite{paper2}, with high probability no less than $1-\frac{C_1}{m}-e^{-C_2n}$, we have
\begin{align*}
    \frac{\langle \text{Hess}f(\bZ)[\xi], \xi \rangle}{\|\xi\|_F^2} &\leq \frac{3(1+\delta_1)(\|\bZ\|_F\|\bX\|_F+2\langle \bX, \bZ\rangle)-(3-\delta_1)\|\bX\|_F^2}{\|\bZ\|_F\|\bX\|_F+\langle \bX,\bZ\rangle}\\
    &\leq \frac{3(1+\delta_1)(\frac{1}{3}+\delta_0+2\delta_0(\frac{1}{3}+\delta_0))-(3-\delta_1)}{(\frac{1}{3}+\delta_0)+\delta_0(\frac{1}{3}+\delta_0)} := \Lambda(\delta_0, \delta_1).
\end{align*}
If $\delta_0 < \frac{1}{6}$, $\delta_1 < \frac{5}{36}$, then we get $\Lambda(\delta_0, \delta_1) < -1$.
\end{proof}

%The above results put the saddle points of phase retrieval completely under control. Specifically, Lemma \ref{lemma:finite} confirms that phase retrieval almost surely has only finite number of critical points. We know from  Theorem 2.2 in \cite{PhaseRetrieval2} that apart from the unique global minimum, all other critical points almost surely lie in the neighborhood of the ``hyper ring''. 
\rev{The above results give us a good idea of the critical points in the ``hyper ring'' region $\{\frac{1}{3}-\delta_0 \le\frac{\|\bZ\|_F}{\|\bX\|_F}\le\frac{1}{3}+\delta_0\}$ on the manifold. Specifically, Lemma \ref{lemma:finite} tells us that there are only a finite number of critical points, and Theorem \ref{thm:prnegative} asserts that these critical points are all strict saddles on the manifold since they have a common negative curvature direction. We are particularly interested in the ``hyper ring'' region because Theorem 2.2 of \cite{PhaseRetrieval2} shows (with a slightly modified objective function) that all the critical points lie in this region with high probability, except the unique global minimum. From Theorem \ref{thm:escape}, we now know that the PGD will avoid saddles and converge to the global minimum.
}%\footnote{\rev{\cite{PhaseRetrieval2} uses a truncated objective function instead of the original one and works in the Euclidean space. It is doable to prove similar results with the un-truncated objective function on the manifold.} would show that apart from the unique global minimum, all other critical points lie in the neighborhood of the ``hyper ring'' with high probability.} And Theorem \ref{thm:prnegative} asserts that these critical points are all strict saddles on the manifold as they have a common negative curvature direction.   

\begin{figure}[htbp]
\centering
\begin{subfigure}[t]{0.48\textwidth}
\centering
\includegraphics[width=6cm]{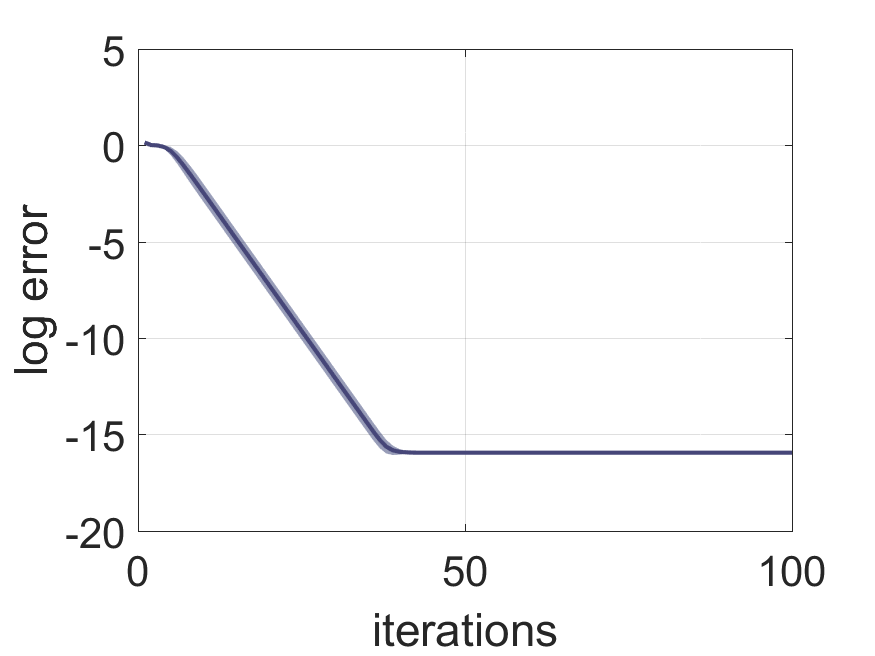}
\caption{$log10$ error for the mean case}
\label{fig:pr1}
\end{subfigure}
\begin{subfigure}[t]{0.48\textwidth}
\centering
\includegraphics[width=6cm]{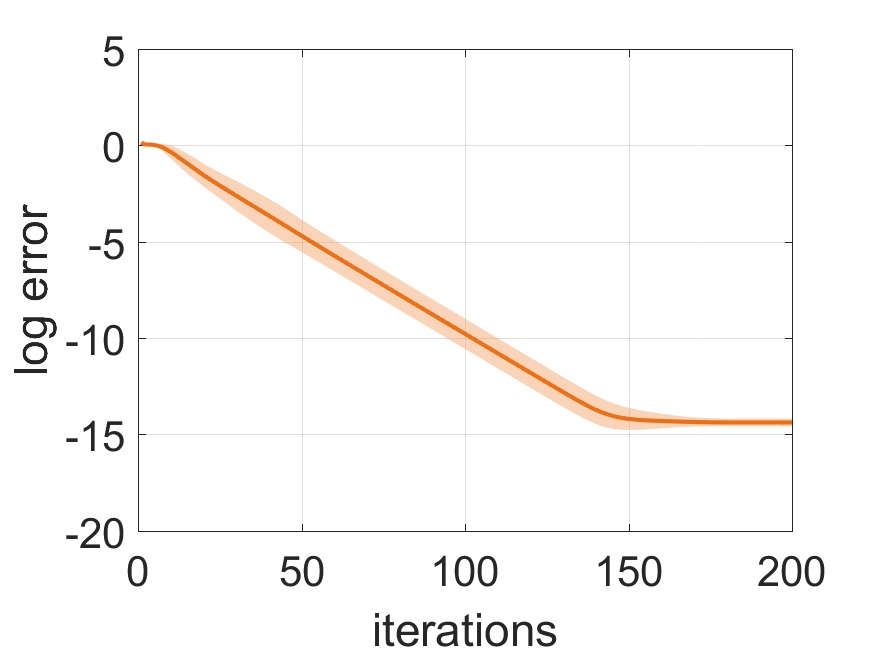}
\caption{$log10$ error for a specific realization}
\label{fig:pr2}
\end{subfigure}
\caption{Convergence (visualized as error band) of PGD for phase retrieval}
\label{fig:pr}
\end{figure}
Figure \ref{fig:pr} shows the $log10$ error convergence of the PGD for phase retrieval on the manifold $\M_1$. The left figure is about the mean case, also called the population problem, while the right one is a specific case with a certain group of $\{A_j\}_{j=1}^m$, where $m=12n$. In both experiments, we take $n = 256$, learning rate $\alpha=\frac{1}{3}$, draw $100$ $z_0$ from i.i.d. Gaussian distribution ($\bZ_0=\bz_0\bz_0^\top$), and minimize $\mathbb{E}f(\bZ)$ or $f(\bZ)$ starting from these random initializations. The darker central line is the average and the band shows the deviation. In general, it can be seen that the PGD is hardly affected by the possible existence of saddle points and converges to the minimum. 

This experiment has also demonstrated the curious phenomenon mentioned at the beginning of Section \ref{sec:Mr}, namely a first order method such as PGD converges exponentially fast (i.e. linearly), even though in the Euclidean space it does not (i.e. only sublinearly). This will be explained in the upcoming work \cite{paper2}.

%Figure \ref{fig:pr} shows the log10 error convergence of PGD for phase retrieval on the manifold $\M_1$ for (a) phase retrieval population problem where $\mathbb{E}f(\bZ)= \frac{3}{2}\|\bZ\|_F^2 + \frac{3}{2}\|\bX\|_F^2 - \|\bZ\|_F\|\bX\|_F - 2\langle\bZ, \bX \rangle$. In this numerical experiment, $N=256$ and learning rate $\alpha=\frac{1}{3}$, we design a random $\bX$ to be recovered, and i.i.d draw 100 $\bZ_0$ from random initialization. (With $\bZ_0=\bz_0\bz_0^\top$, entries of $\bz_0$ from standard normal or uniform distribution.) Here, we display the error reduction of 100 independent experiments, dark central line means the error mean and the band implies the deviation. From previous analysis, it would have its saddle sets to be connceted but form a lower dimensional submanifold which can be escaped. Similarly, in figure(b), we design a standard phase retrieval problem with gaussian measurements. Here, $N=256$ and $m=12N$, learning rate $\alpha=\frac{1}{3}$. By former results, we would know in this case the saddle sets can only contains finite many strict saddle which can also be escaped. Generally, it can be seen that the algorithm is hardly affected by the possible existence of saddle points, and converges exponentially fast. 
%\toadd{(also need a more readable figure)}

\section{Applications beyond the low-rank matrix manifold}
\label{sec:app}

As is mentioned in Section \ref{sec:PGD}, although our primary setting is the low-rank matrix manifold $\overline{\M_r}$, the asymptotic convergence to the minimum and escape of strict saddles (strict critical submanifolds) is valid on arbitrary finite dimensional Riemannian manifold $\M$. In particular, the properties of the PGD are well preserved if the manifold is embedded in a Banach space and inherits its metric. Below we discuss the optimization on the unit sphere and the Stiefel manifold as two examples of applications. 

\subsection{Variational eigen problem on a sphere}

Consider $\M = \S^{n-1}$, the sphere embedded in the Euclidean space $\R^n$. We consider the following eigenvalue problem:
\begin{align*}
    g(z) = \lambda z, \quad z\in\R^n.
\end{align*}
Note that $g(z)$ may or may not be linear in $z$. Assume that it admits eigenpairs $(\lambda_1, v_1), (\lambda_2, v_2), \ldots, (\lambda_k, v_k)$, $0 < \lambda_1 < \lambda_2 \le \ldots \le \lambda_k$. If $g(z) = \nabla f(z)$ for some function $f(z)$, then to find $(\lambda_1, v_1)$ is to solve the following optimization problem:
\begin{align*}
    \min_z f(z) \quad  \text{s.t. } z \in \M=\S^{n-1}.
\end{align*}

Viewed as an embedded Riemannian manifold, the tangent space, tangent space projection and retraction on $\M=\S^{n-1}$ are given as follows:
\begin{align*}
    T_z &= \{\xi\in\R^n: \xi^{\top}z=0\}, \\
    P_{T_z} &= I-zz^{\top}, \\
    R(y) &= \frac{y}{\|y\|_2}.
\end{align*}
Note that $R(y)$ is a second-order retraction, because for any $z\in\M$, $\xi\in T_z$, we have
\begin{align*}
    R(z+\alpha \xi) = \frac{z+\alpha \xi}{\|z+\alpha \xi\|_2} = (z+\alpha \xi)(1+\alpha^2\|\xi\|_2^2)^{-\frac{1}{2}} = z+\alpha \xi + \mathcal{O}(\alpha^2).
\end{align*}
The Levi-Civita connection on $\M$ is the projection of the Levi-Civita connection of the ambient space (which is the directional gradient in $\R^n$)
\begin{align*}
    \widetilde{\nabla}_{\xi_z} \eta = P_{T_z}(\nabla_{\xi_z} \eta) = (I-zz^{\top})(\nabla_{\xi_z} \eta), \quad \eta\in T_{\M}, \quad \xi_z \in T_{z}.
\end{align*}
The Riemannian gradient on $\M$ is 
\begin{align*}
    \text{grad} f(z) = P_{T_z}(\nabla f(z)).
\end{align*}
So $z$ is a critical point on $\M$ if and only if $z$ is an eigenvector of the eigen problem $g(z) = \lambda z$.
The Riemannian Hessian on $\M$ is 
\begin{align*}
    \text{Hess}f(z)[\xi] =  P_{T_z}(\nabla^2 f(z)[\xi]) - (z^\top\nabla f(z))\xi.
\end{align*}

If $g(z)$ is linear in $z$, then $f(z)$ is quadratic. With the positiveness assumption, we have $f(z) = z^\top A z$, where $A$ is an SPD matrix. Then $f(x_i) = \lambda_i$, $\text{grad} f(z) = Az -(z^\top A z)z$, and $\xi^\top\text{Hess}f(z)[\xi] = \xi^\top A \xi - (z^\top A z) \xi^\top \xi$. It is easy to see that $v_1$ is the unique (up to sign) global minimum with a positive Hessian, and $v_s (s>1)$ are all strict saddles whose Hessian has at least one negative curvature direction $\xi = v_s$. 

It is interesting to look at the case where a non-minimal eigenvalue has multiplicity greater than 1. Assume that $\lambda_s = \lambda_{s+1} = \ldots = \lambda_{s+t}$, then the submanifold $\N = \{y\in\R^n \mid y = c_sv_s + \ldots + c_{s+t}v_{s+t}, \quad c_s^2+ \ldots + c_{s+t}^2 = 1\}$ is an immersed submanifold of $\M$, and it is a strict critical submanifold of $f$ if $s\ge 2$. Since the number of such submanifolds is finite, escape from these submanifolds towards $x_1$ is ensured by the tools in Section \ref{sec:submanifold}.

When $g(z)$ is not linear in $z$, as $f(z)$ now contains non-quadratic terms, it is not immediately clear from the algebraic expression whether $\text{Hess}f(x_s), s>1$ has negative curvature direction, though it can be verified numerically. 

The first numerical example is from the discretized 1D Schr\"odinger eigen problem $-\Delta u + V(x) u = \lambda u$ with periodic boundary condition, where $V(x)$ is taken to be the smoothed 1D Kronig-Penney (KP) potential describing free electrons in 1D crystal \cite{kronig1931quantum}\cite{ozolicnvs2013compressed}. Figure \ref{fig:1-1} shows the profile of the KP potential defined on $D = [0,50]$ with 5 energy wells and periodic BC. Figure \ref{fig:1-2} shows the  first 30 eigenvalues of the operator $-\Delta+V(x)$. We can see that the first 5 eigenvalues are clustered (but not identical).

We discretize $D$ into $n=2^7$ grids and solve the discretized problem on $\M = \S^{n-1}$ with the PGD starting from a random initialization. The step size is $\alpha=0.01$. In Figure \ref{fig:1-3}, we observe that the generated point series first seem to ``stagnate'' near a non-minimal eigen state, but then escape and converge towards the minimum. Figure \ref{fig:1-4} shows the profile of the computed ground energy state $v_1$, which is quite close to the true ground state but slightly deformed. An improvement will be proposed in the next subsection.

\begin{figure}[ht]
    \centering
    \begin{subfigure}[t]{.40\linewidth}
        \includegraphics[width = \linewidth]{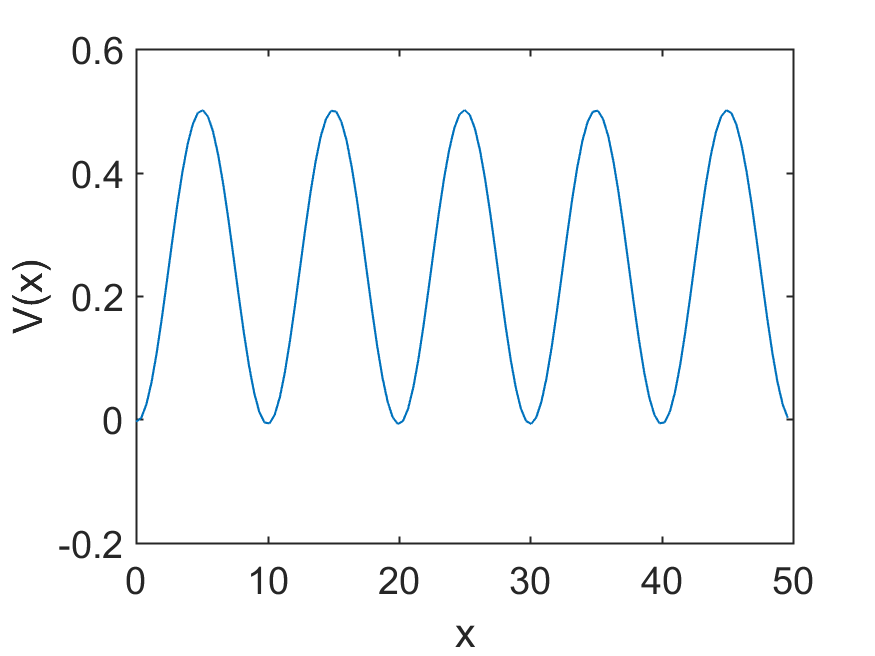}
        \caption{Profile of $V(x)$}
        \label{fig:1-1}
    \end{subfigure}
    \hspace{0.08\textwidth}
    \begin{subfigure}[t]{.40\linewidth}
        \includegraphics[width = \linewidth]{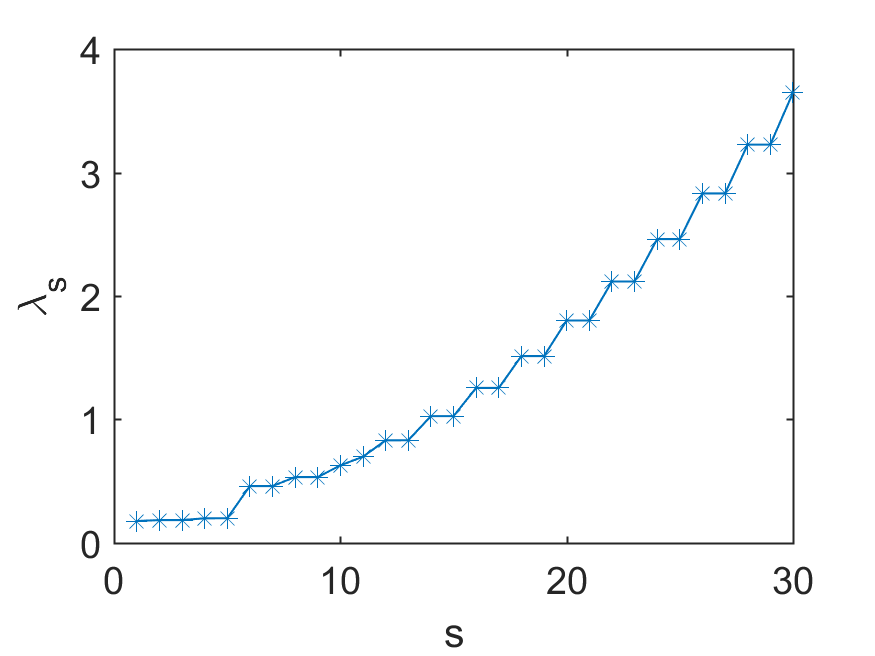}
        \caption{First 30 eigenvalues of $-\Delta+V(x)$}
        \label{fig:1-2}
    \end{subfigure}
    \begin{subfigure}[t]{.40\linewidth}
        \includegraphics[width = \linewidth]{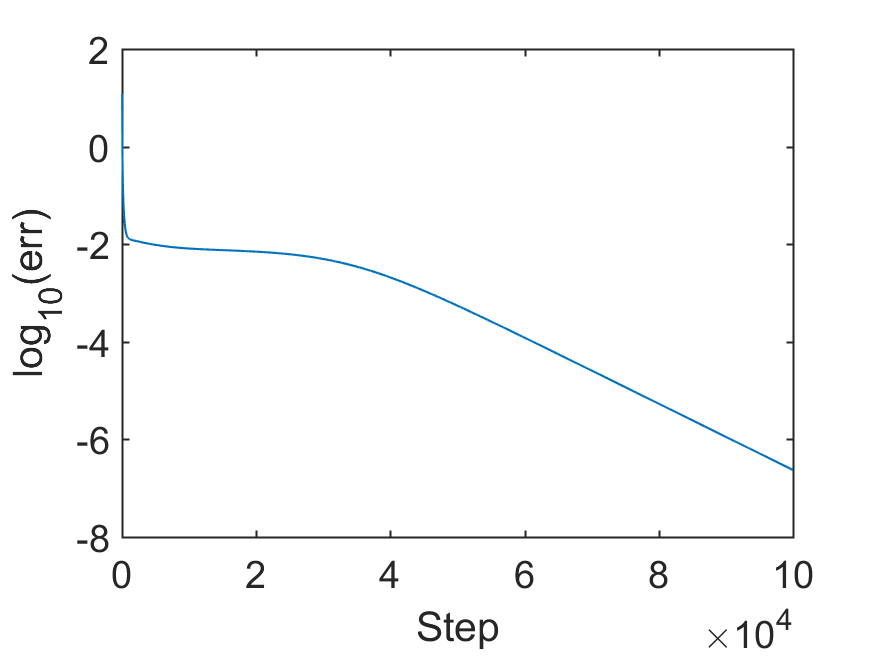}
        \caption{Error decay in PGD}
         \label{fig:1-3}
    \end{subfigure}
    \hspace{0.08\textwidth}
    \begin{subfigure}[t]{.40\linewidth}
        \includegraphics[width = \linewidth]{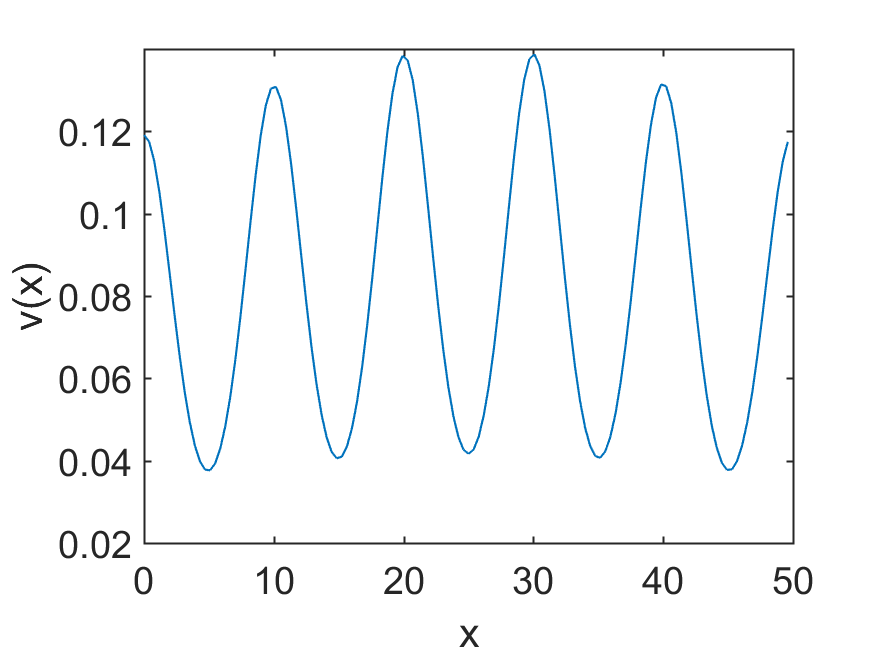}
        \caption{Profile of the first eigenstate $v_1$}
        \label{fig:1-4}
    \end{subfigure}
    \caption{Solving the linear Schr\"odinger eigen problem on the sphere}
    \label{fig:1}
\end{figure}

\vspace{6pt}

The second example is the nonlinear Schr\"odinger eigen problem $-\Delta u + V(x) u + \beta |u|^2u = \lambda u$, or the so-called Gross-Pitaevskii eigenvalue problem for the Bose-Einstein Condensate (BEC) \cite{pitaevskii2016bose}. It gives a more accurate description of the dynamics of Bosonic gases at ultra low temperature. With the presence of the nonlinear term $\beta |u|^2u$, linear eigensolvers would fail, and the optimization of its variational form becomes the state-of-art solver, see e.g. \cite{henning2020sobolev}. Apart from the PGD based on the $L^2$ metric, there can be other PGD algorithms based on other types of metrics and with different convergence theories, whose analysis is beyond the scope of this paper.

We use the same potential function $V(x)$ and discretization size as above. The nonlinear term has the weight $\beta = 1$. The objective function is
\begin{align*}
    f(z) = \frac{1}{2}z^\top Az + \frac{\beta}{4} \sum_{j=1}^n z(j)^4, \quad A = -L+V.
\end{align*}
For an eigenstate $v_s$, the eigenvalue associated to it is 
\begin{align*}
    \lambda_s = 2f(s) + \frac{\beta}{2} \sum_{j=1}^n z(j)^4.
\end{align*}
We compute the first two eigenstates of the nonlinear Schr\"odinger problem using the PGD with stepsize $\alpha=0.01$. Figure \ref{fig:2-3} shows their profiles. Figures \ref{fig:2-1} and \ref{fig:2-2} demonstrate the convergence of the PGD towards the computed eigenvalues. 

To verify that $v_2$ is a strict saddle point, we numerically compute the smallest eigenvalue of $\text{Hess}f(v_2)$ and  $\lambda_{\text{min}}(\text{Hess}f(v_2)) = -0.0024 < 0$. Figure \ref{fig:2-4} shows a profile of the corresponding eigenvector $\xi_{\text{min}}$, i.e. a negative curvature direction.

\begin{figure}[ht]
    \centering
    \begin{subfigure}[t]{.40\linewidth}
        \includegraphics[width = \linewidth]{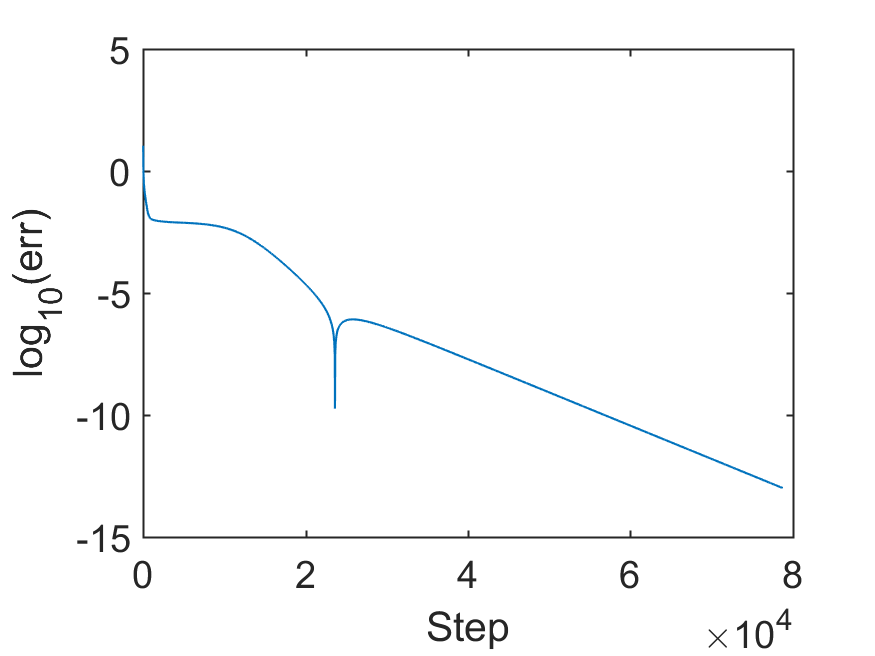}
        \caption{Error decay of PGD when \\ computing $v_1$}
        \label{fig:2-1}
    \end{subfigure}
    \hspace{0.08\textwidth}
    \begin{subfigure}[t]{.40\linewidth}
        \includegraphics[width = \linewidth]{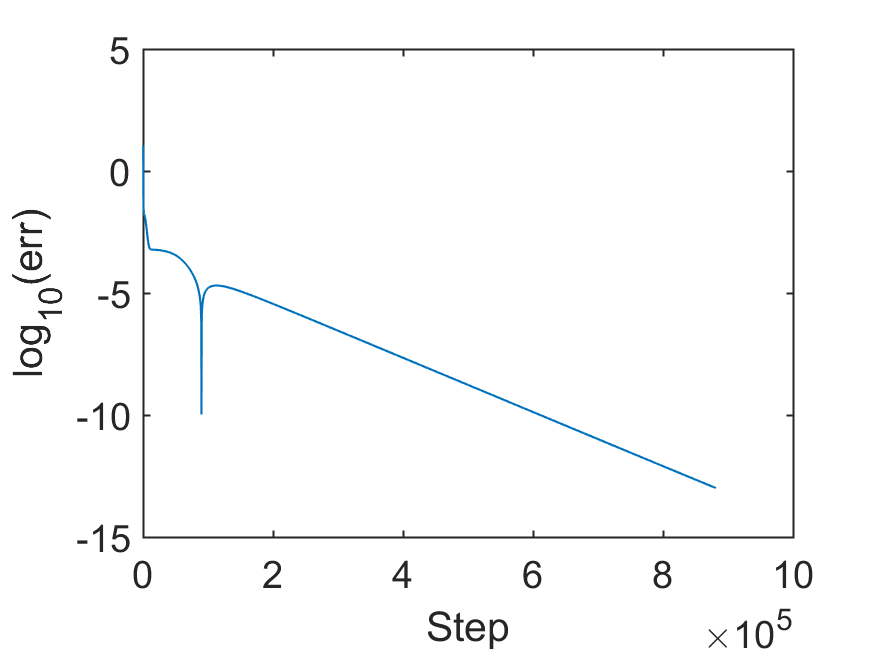}
        \caption{Error decay of PGD when \\ computing $v_2$}
        \label{fig:2-2}
    \end{subfigure}
    \begin{subfigure}[t]{.40\linewidth}
        \includegraphics[width = \linewidth]{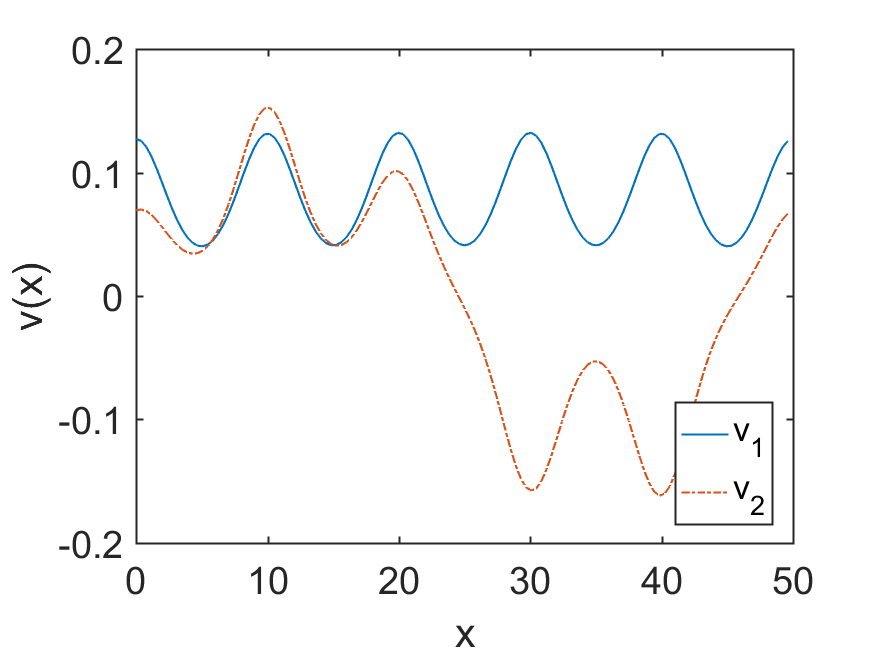}
        \caption{Profile of eigenstates $v_1$ and $v_2$}
        \label{fig:2-3}
    \end{subfigure}
    \hspace{0.08\textwidth}
    \begin{subfigure}[t]{.40\linewidth}
        \includegraphics[width = \linewidth]{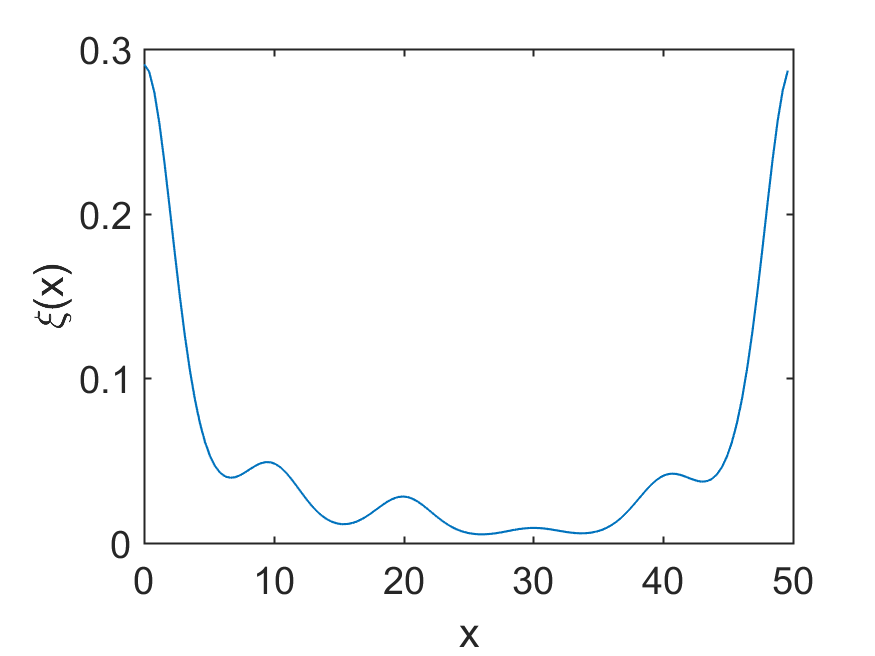}
        \caption{A negative curvature direction}
        \label{fig:2-4}
    \end{subfigure}
    \caption{Solving the nonlinear Schr\"odinger eigen problem on the sphere}
    \label{fig:2}
\end{figure}

%\vspace{6pt}

Apart from physical problems like BEC eigenstates, linear and nonlinear eigenproblems also find applications in image processing and machine learning. For example, the MaxCut problem corresponds to a linear eigenproblem, while the optimization of the Ginzburg-Landau type functional in image segmentation and learning tasks corresponds to a nonlinear eigenproblem, see e.g. \cite{bertozzi2012diffuse}\cite{hou2019fast}. Although there are many algorithms tailored for linear eigenproblems, their nonlinear relatives often lack a rigorous convergence guarantee. Manifold optimization thus provides a more versatile point of view for them.

\subsection{Simultaneous eigen solver on the Stiefel manifold}
Subspace iteration is a common technique for accelerating the convergence of smallest eigenstates in linear eigen problems, especially when the ground states are clustered, as is in the previous examples. 

From the viewpoint of manifold optimization, to solve the first $m$ eigenstates simultaneously can be posed as the optimization on the Stiefel manifold $\M = \{\bZ \in \R^{n\times m}: \bZ^\top\bZ = I_m\}$:
\begin{align*}
    \min_{\bZ} \text{ trace}(f(\bZ)) \quad \text{s.t. } \bZ \in \M = \{\bZ \in \R^{n\times m}: \bZ^\top\bZ = I_m\}.
\end{align*}

The Stiefel manifold \cite{edelman1998geometry}\cite{kaneko2012empirical} is the set of all $m$-frames in $\R^n$. When $m=1$, it reduces to the sphere $\S^{n-1}$. With the Euclidean metric, its tangent space, tangent space projection and retraction are given as follows:
\begin{align*}
    &T_{\bZ} = \{\xi\in \R^{n\times m}: \xi^\top\bZ + \bZ^\top\xi = 0 \}, \\
    &P_{T_{\bZ}}(\bf{Y}) = \bf{Y} - \bZ \text{ sym}(\bZ^\top\bf{Y}), \\
    &R(Y) = \text{qf }(Y),
\end{align*}
where sym takes the symmetric part and qf takes the $Q$ factor of QR decomposition. Similar to the case of the sphere manifold, the Riemannian connection and gradient are defined by the projection onto the tangent space. When $f(\bZ) = \bZ^\top A\bZ$, we have
\begin{align*}
    \text{grad} f(\bZ) &= P_{T_{\bZ}}(A\bZ), \\
    \langle\xi, \text{Hess} f(\bZ)[\xi] \rangle &= \text{tr}( \xi^\top A\xi - (\xi^\top\xi)(\bZ^\top A\bZ)).
\end{align*}
It is easily verified that the minimum is achieved when $span \bZ = span \{v_1, \ldots, v_m\}$, and all $\bZ$ that span other eigen subspaces are strict saddles if all the eigenvalues are distinct.

We compute the first 5 eigenstates simultaneously for the linear Schr\"odinger eigen problem with the same potential as in Figure \ref{fig:1-1}. The step size is $\alpha = 0.01$. Figures \ref{fig:3-1} and \ref{fig:3-2} compare the computed eigenstates extracted from $\bZ$ and the true eigenstates, which are almost identical. In Figure \ref{fig:3-3}, we can see that the subspace iteration on the Stiefel manifold achieves much better convergence in fewer steps than the optimization on the sphere.

\begin{figure}[ht]
    \centering
    \begin{subfigure}[t]{.40\linewidth}
        \includegraphics[width = \linewidth]{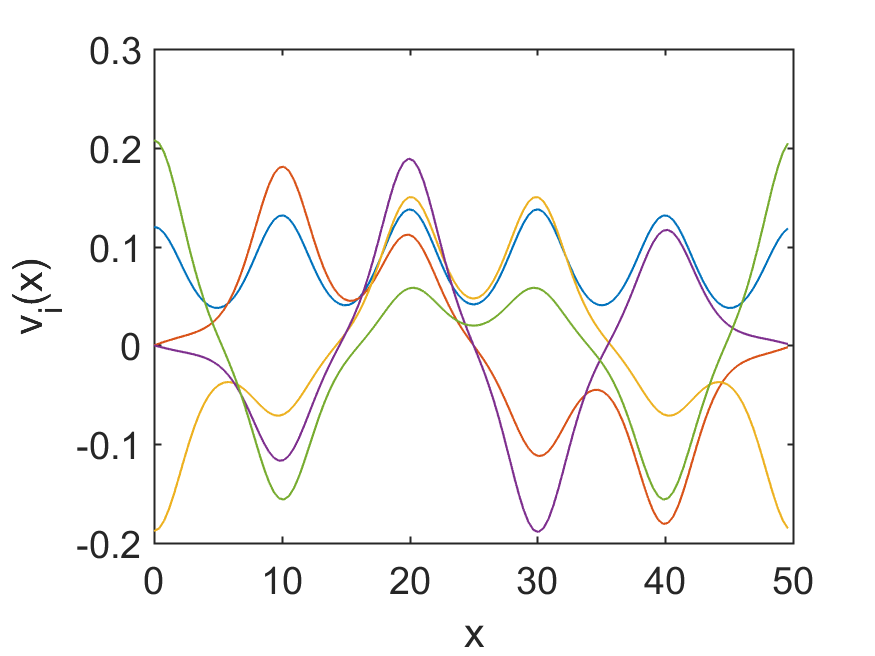}
        \caption{Computed eigenstates}
        \label{fig:3-1}
    \end{subfigure}
    \hspace{0.08\textwidth}
    \begin{subfigure}[t]{.40\linewidth}
        \includegraphics[width = \linewidth]{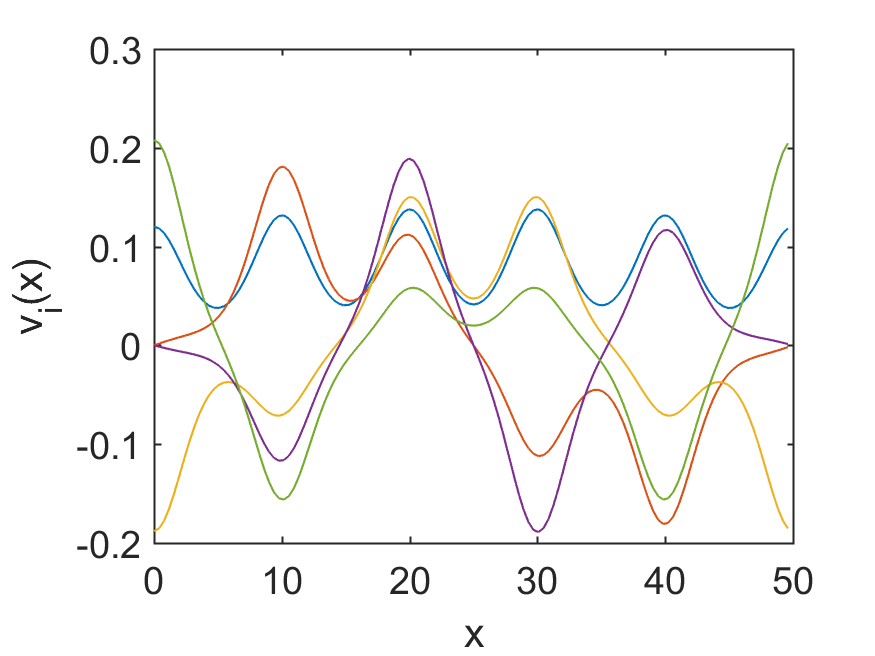}
        \caption{True eigenstates}
        \label{fig:3-2}
    \end{subfigure}
    \begin{subfigure}[t]{.40\linewidth}
        \includegraphics[width = \linewidth]{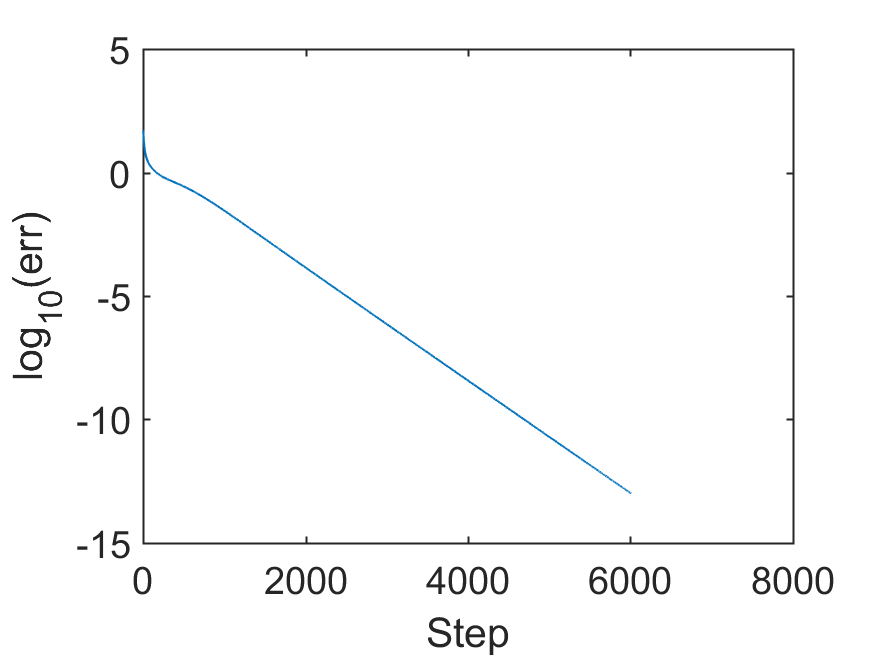}
        \caption{Error decay of PGD}
        \label{fig:3-3}
    \end{subfigure}
    \hspace{0.08\textwidth}
    \begin{subfigure}[t]{.40\linewidth}
        \includegraphics[width = \linewidth]{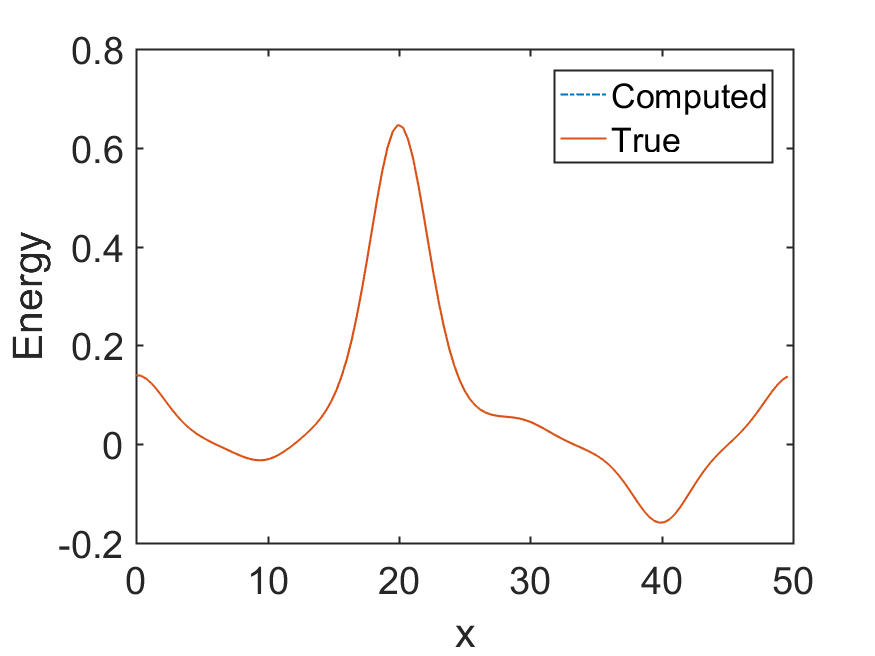}
        \caption{Accumulated energy of the first 5 eigenstates}
        \label{fig:3-4}
    \end{subfigure}
    \caption{Simultaneously solving the first 5 eigenstates of the linear Schr\"odinger problem on the Stiefel manifold}
    \label{fig:3}
\end{figure}

Application of the Stiefel manifold optimization can also be extended to data science, e.g. frame construction and dictionary learning \cite{cai2014data}\cite{sun2016complete}, if the frame/dictionary satisfies orthonormal assumptions.

\section{Conclusion and future works}
\label{sec:conclusion}
We have studied the asymptotic escape of strict saddles points of the PGD on the Riemannian manifolds. The first main contribution of this paper is that it pushes the boundary of current analysis to non-isolated saddle sets, proving when the PGD can escape and indicating when it cannot. As a general tool, it can be applied to various settings as long as the manifold of interest satisfies certain smoothness conditions. This is demonstrated by several representative examples from different fields. 

The saddle analysis of phase retrieval on the low-rank matrix manifold serves as an application of the above asymptotic escape result, but it also stands as an insightful result by itself. We have shown that it always has a finite number of critical points, and the saddles are strict saddles with high probability. Essentially, the low-rank matrix manifold sheds light on the intrinsic quadratic (instead of quartic) structure of this problem.

In addition to the asymptotic convergence behavior of the PGD, the convergence rate is also an important issue. Empirical linear convergence rates in many low-rank matrix recovery problems are already observed but yet to be explained. This will be the topic of our future work \cite{paper2}, where we prove the linear convergence rate using the quadratic nature of those problems on the manifold $\overline{\M_r}$.

%\toadd{(Add more future applications?)}

\vspace{8pt}

{\textbf{Acknowledgements.}} The research was in part supported by NSF Grants DMS-1613861, DMS-1907977 and DMS-1912654. \rev{We would like to thank the anonymous referees for their insightful comments, which improve the quality of our manuscript.}

% References
\bibliographystyle{siamplain}
\bibliography{references}

\appendix{
\appendixnotitle
}
As we mentioned in Section \ref{sec:submanifold}, when there are a bunch of self-connected critical submanifolds (generalization of critical points), the  escape of strict critical submanifolds (generalized strict saddles) and convergence to a minimum rely on the number or the structure of such critical submanifolds. \rev{When the number is uncountable, the situation can be quite complicated.} 

In this appendix, we discuss some structural properties of critical submanifolds that may help untangle their successive relations. We introduce the concepts of index and transversality, point out the transversality properties of certain functions and their consequences, and link the stable manifolds of the gradient flow to that of the gradient descent.

\begin{definition}[Index]
    For $f: \M\mapsto\R$, let $p$ be a critical point of $f$, then the \emph{index} of $p$ is
    \begin{align*}
        \lambda_p := \text{dim }T^u_p\M.
    \end{align*}
\end{definition}

\begin{remark}
    All critical points in the same connected critical submanifold $\mathcal{N}$ have the same index, which is defined as the index $\lambda_\mathcal{N}$ of the submanifold $\mathcal{N}$. An equivalent way to define strict critical submanifold is $\lambda_\mathcal{N}>0$.
\end{remark}

\begin{definition}[Transversality]
    (1) For smooth maps $f: \mathcal{N}_1 \mapsto \M$ and $g: \mathcal{N}_2 \mapsto \M$, we say that $f$ is \emph{transverse} to $g$, iff for any $X_1, X_2$ such that $f(X_1) = g(X_2) = Y$, 
    \begin{align*}
        df(T_{X_1}\mathcal{N}_1) + dg(T_{X_2}\mathcal{N}_2) = T_Y\M,
    \end{align*}
    where $df$ and $dg$ are gradient vector fields of $f$ and $g$; \\
    (2) If $\mathcal{N}_1$ and $\mathcal{N}_2$ are immersed submanifolds of $\M$, then $\N_1$ is \emph{transverse} to $\mathcal{N}_2$ iff for any $X\in\mathcal{N}_1\cap\mathcal{N}_2$,
    \begin{align*}
        T_X\mathcal{N}_1 + T_X\mathcal{N}_2 = T_X\M.
    \end{align*}
    Two immersed submanifolds vacuously transverse if they do not intersect.
\end{definition}

\begin{remark}
    A function $f: \M\mapsto\R$ is called \emph{Morse-Bott} if all its critical points lie in some disjoint union of connected and nondegenerate critical submanifolds; $f$ is called \emph{Morse-Smale} if it satisifies the Morse-Smale transversality condition, i.e. for any two critical submanifolds $\mathcal{N}_1$, $\mathcal{N}_2$, their stable and unstable manifolds intersect transversally.
\end{remark}

The transversality condition for immersed manifolds simply means that two manifolds ``cross'' each other and do not ``overlap''. Figure \ref{fig:trans} is a vivid illustration of transversality on a 2-dimensional manifold. If the objective function $f$ is a Morse-Smale function, transversality implies more favorable properties.

\begin{figure}[ht]
    \centering
    \includegraphics[width = 0.4\textwidth]{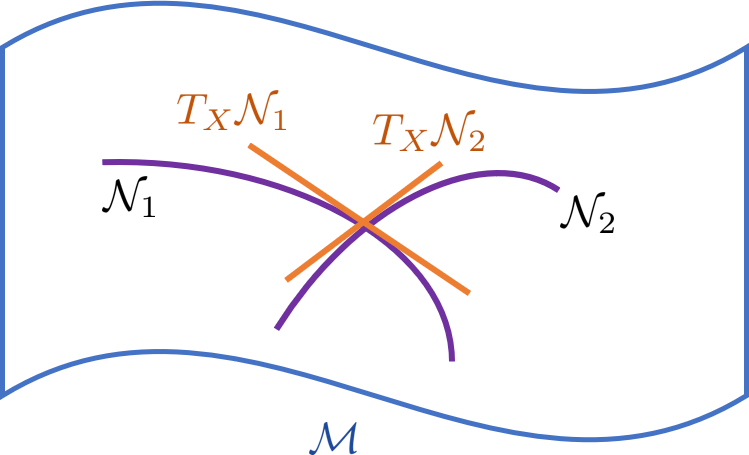}
    \caption{An illustration of transversality}
    \label{fig:trans}
\end{figure}

\begin{theorem}[Corollary 6.27 in \cite{banyaga2013lectures}]
\label{thm:MS1}
    For a Morse-Smale function $f$, any critical point $p$ of $f$ satisfies
    \begin{align*}
        \overline{W^u(p)} &= \bigcup_{p\succeq q} W^u(q), \\
        \overline{W^s(p)} &= \bigcup_{r\succeq p} W^s(r),
    \end{align*}
    where $W^s(p)$ (resp. $W^u(p)$) is the stable (resp. unstable) manifold of $p$ defined by gradient flow line, and $p\succeq q$ means $W^u(p)\cap W^s(q) \neq \emptyset$.
\end{theorem}

\begin{theorem}
\label{thm:MS2}
    %(a useful theorem: for a Morse-Bott function, if two separate critical submanifolds have the same index (negative dimension), then they vacuously transverse, i.e. one's unstable manifold and the other's stable manifold does not intersect. (i.e. same index -- won't be trapped between))   
    For a Morse-Smale function $f$, if two critical submanifolds  $\mathcal{N}_1$ and $\mathcal{N}_2$ have the same index, then they vacuously transverse, i.e. $W^u(\mathcal{N}_1) \cap W^s(\mathcal{N}_2) = \emptyset$.
\end{theorem}
\begin{proof}
    By Proposition 6.2 in \cite{banyaga2013lectures}, if $W^u(\mathcal{N}_1) \cap W^s(\mathcal{N}_2) \neq \emptyset$, then their intersection is an embedded submanifold of dimension $(\lambda_{\mathcal{N}_1}-\lambda_{\mathcal{N}_2})$. But $\lambda_{\mathcal{N}_1}-\lambda_{\mathcal{N}_2}=0$, which is a contradiction.
\end{proof}

%Theorem \ref{thm:MS1} points out a way to take unions of stable manifolds of infinitely many critical submanifolds by looking at their successive relations. This could facilitate the proof that the union of stable manifolds is of measure 0. On the other hand, Theorem \ref{thm:MS2} helps rule out the possibility that a gradient flow line escapes one saddle but falls into another, and gets trapped between forever. 

 {Both Theorem \ref{thm:MS1} and Theorem \ref{thm:MS2} are helpful when taking the union of stable manifolds of infinitely many critical submanifolds. Theorem \ref{thm:MS1} shows that the closure of the stable/unstable manifold of one critical set is the union of the stable/unstable manifolds of the sets that have successive relations with it. On the other hand, Theorem \ref{thm:MS2} shows that the successive relations are strictly limited by the indices (i.e. negative curvature dimensions) of the critical sets. This successive relation simply cannot happen between sets of the same index.

It should be stressed that the above results are on the stable/unstable manifold of \emph{gradient flows}, not \emph{gradient descents}. Whether this can be generalized to gradient descents is still unclear. 
We know that with first-order retraction property, as $\alpha\to0$, the projected gradient descent on the manifold approximates the gradient flow line. It can be proved that the respective stable/unstable manifolds also converge, as long as the retraction is at least first-order and the domain is compact. However, the transversality concerns the ``angles'' at the intersection of these submanifolds. Even the uniform convergence of submanifolds cannot ensure the preservation of their intersection angles along the convergence. 

This discussion aims to draw interest to the vast possibilities that Morse theory has to offer. They point out a way to deal with complex geometries of critical point sets.
We plan to conduct further studies along this direction to quantify the above relations.}

%\vspace{6pt}

\end{document}